\documentclass[12pt]{amsart}
\usepackage{amsmath} 
\usepackage{amssymb}
\usepackage{amsfonts}
\usepackage{ndstye}

\newcommand{\br}{\mathbf R}

\newcommand{\Cal}{\mathcal}

\newcommand{\gs}{\gtrsim}

\newcommand{\hess}{\operatorname{Hess}}

\renewcommand{\iff}{\Leftrightarrow}
\newcommand{\im}{\operatorname{Im}}

\newcommand{\ls}{\lesssim}

\newcommand{\mn}[1]{\Vert#1\Vert}

\newcommand{\ol}{\overline}

\newcommand{\re}{\operatorname{Re}}

\newcommand{\restr}[1]{\big|_{#1}}

\newcommand{\set}[1]{\left\{\,#1\,\right\}}

\newcommand{\st}{\Sigma _2}
\newcommand{\supp}{\operatorname{\rm supp}}

\newcommand{\w}[1]{\langle #1\rangle }

\newcommand{\wf}{\operatorname{WF}}

\newcommand{\wt}{\widetilde}

\newcommand{\se}[1]{ S^#1_{1-{\varepsilon},{\varepsilon}}}

\numberwithin{equation}{section}

\begin{document}

%\hfill  {\today}\qquad File: LimP.tex 

%\vskip 10mm

\baselineskip 18pt %ev 12 17 18 works
\lineskip 2pt
\lineskiplimit 2pt

\title[Limit bicharacteristics]{Solvability and limit bicharacteristics}
\author[NILS DENCKER]{{\textsc Nils Dencker}}
\address{Centre for Mathematical Sciences, University of Lund, Box 118,
SE-221 00 Lund, Sweden}
\email{dencker@maths.lth.se}

\date{26 juni 2016} %\today 


\maketitle

\thispagestyle{empty}

\section{Introduction}

We shall consider the solvability for a classical pseudodifferential
operator $P$ on a $C^\infty$ manifold $X$.  This means that $P$ has an
expansion $p_m + p_{m-1} + \dots$ where $p_k \in S^{k}_{hom}$ is
homogeneous of degree $k$ and $p_m = {\sigma}(P)$
is the principal symbol of the operator. The operator $P$ 
is solvable at a compact set $K \subseteq X$ if the equation  
\begin{equation}\label{locsolv}
Pu = v 
\end{equation}
has a local solution $u \in \Cal D'(X)$ in a neighborhood of $K$
for any $v\in C^\infty(X)$
in a set of finite codimension.  
We can also define the microlocal solvability at any compactly based cone
$K \subset T^*X$, see ~Definition~\ref{microsolv}. 

A pseudodifferential operator is of
principal type if the Hamilton vector field 
\begin{equation}
H_p = \sum_{j=1}^{n}\partial_{{\xi}_j} p\partial_{x_j} -
\partial_{x_j} p\partial_{\xi_j} 
\end{equation}
of the principal symbol $p$ does not have the radial
direction $\w{{\xi}, \partial_{\xi}}$ at $p^{-1}(0)$, in particular
$H_p \ne 0$ then. By homogeneity $H_p$ is well defined 
on the cosphere bundle $S^*X = \set{(x,{\xi}) \in T^*X:\ |{\xi}| =
  1}$, defined by some choice of Riemannean metric.
For pseudodifferential operators of principal type, it is known
\cite{de:nt} \cite{ho:nec} that  
local solvability at a point is equivalent to condition (${\Psi}$)
which means that
\begin{multline}\label{psicond} \text{$\im (ap)$
    does not change sign from $-$ to $+$}\\ 
 \text{along the oriented
    bicharacteristics of $\re (ap)$}
\end{multline}
for any $0 \ne a \in C^\infty(T^*M)$. Oriented bicharacteristics are
the positive flow-outs of the Hamilton vector field $H_{\re (ap)} \ne
0$ on $\re (ap) =0$. Bicharacteristics of $\re ap$ are also called
semi-bicharacteristics of $p$.

We shall consider the case when the principal symbol is real and
vanishes of at least second order at an involutive manifold~$\st$,
thus~ $P$ is not of principal type. For operators which
are not of principal type, the values of the subprincipal symbol
$p_{m-1}$ at ~$\st$ becomes important.  
In the case when principal symbol $p = {\xi}_1{\xi}_2$, Mendoza and
Uhlman~\cite{MU1} proved that $P$ was not 
solvable if the subprincipal symbol changed sign on the $x_1$ or
$x_2$ lines when ${\xi}_1 = {\xi}_2 = 0$. Mendoza~\cite{Men}
generalized this to the case when the principal symbol vanishes on an
involutive submanifold  having an indefinite Hessian with rank equal to the
codimension of the manifold. The Hessian then gives well-defined limit
bicharacteristis over the submanifold, and $P$ is
not solvable if the subprincipal symbol changes sign on any of these
limit bicharacteristics. This corresponds to condition ($P$) on the
limit bicharacteristics (no sign changes) because in this case one gets
both directions when taking the limits.

In this paper, we shall extend this result to more general
pseudodifferential operators. As in the previous cases, the operator
will have real principal symbol and we shall consider the limits of
bicharacteristics at the set where the principal symbol vanishes of at
least second order. The convergence shall be as smooth
curves, then the limit bicharacteristic also is a smooth curve. We shall
also need uniform bounds on the curvature of the characteristics at
the bicharacteristics, but only along the tangent space of a Lagrangean
 submanifold of the characteristics, which we call a grazing Lagrangean space,
see~\eqref{cond0}. This gives uniform bounds on the linearization of
the normalized 
Hamilton flow on the tangent space of this submanifold at the
bicharacteristics. Our main result is Theorem~\ref{mainthm}, which
essentially says that under these conditions the operator is not
solvable at the limit 
bicharacteristic if the quotient of the imaginary part of 
the subprincipal symbol with the norm of the Hamilton vector field
switches sign from $-$ to $+$ on the
bicharacteristics and becomes unbounded as they converge to the limit
bicharacteristic.

\section{Statement of results}

Let the principal symbol $p$ be real valued, ${\Sigma} =
p^{-1}(0)$ be the characteristics, and ${\Sigma}_2$ be the set of
double characteristics, i.e., the points on ${\Sigma}$ where $dp =
0$. Let $\set{{\Gamma}_j}_{j=1}^\infty$ be a set of bicharacteristics
of $p$ on $S^*X$ so that ${\Gamma}_j \subset {\Sigma}\setminus {\Sigma}_2$
are uniformly bounded in $C^\infty$ when para\-metrized on a uniformly
bounded interval (for example  with respect
to the arc length). These bounds are defined with respect to some choice of 
Riemannean metric on $T^*X$, but different choices of metric will only
change the bounds with fixed constants. In particular, we have a
uniform bound on the arc lengths:
\begin{equation}\label{hpcond0}
|{\Gamma}_j| \le C \qquad \forall\,j
\end{equation}
We also have that ${\Gamma}_j = \set{{\gamma}_j(t):\ t \in I_j}$ with
$|{\gamma}_j'(t)| \equiv 1$ and $|I_j| \le C$, then $|{\gamma}_j^{(k)}(t)|
\le C_k$ for $t \in I_j$ and all $j$, $k \ge 1$. Let  $\wt p =
p/|\nabla p|$ then the normalized Hamilton vector field is equal to
\begin{equation*}
 H_{\wt p} = |H_p|^{-1}H_p \qquad \text{on $p^{-1}(0)\setminus \st$}
\end{equation*}
then ${\Gamma}_j$ is uniformly bounded in $C^\infty$ if
\begin{equation}\label{hpcond}
| H_{\wt p}^k \nabla \wt p| \le C_k \qquad \text{on
  ${\Gamma}_j$ }\quad \forall\, j, k 
\end{equation}
where $\nabla \wt p$ is the gradient of~$\wt p$. 
Thus the normalized
Hamilton vector field $H_{\wt p}$ is
uniformly bounded in $C^\infty$ as a vector field over~ ${\Gamma}$.
Observe that the bicharacteristics
have a natural orientation given by the Hamilton vector field.  
Now the set of bicharacteristic curves
$\set{{\Gamma}_j}_{j=1}^\infty$ is uniformly bounded in $C^\infty$ when
para\-metrized with respect to the arc length, and therefore it is a
precompact set. Thus there exists a subsequence ${\Gamma}_{j_k}$, $k
\to \infty$, that converge to a smooth 
curve ${\Gamma}$ (possibly a point), called a limit bicharacteristic by
the following definition.

\begin{defn}
We say that a sequence of smooth curves ${\Gamma}_j$ on a smooth
manifold converges to a
smooth limit curve ${\Gamma}$ (possibly a point) if there exist
parametrizations on uniformly bounded intervals that converge in $C^\infty$. 
If $p \in C^\infty(T^*X)$, then a smooth curve ${\Gamma} \subset \st\
\bigcap S^*X$ is  
a {\em limit bicharacteristic} of $p$ if there exists
bicharacteristics ${\Gamma}_j$ that converge to it. 
\end{defn}

Naturally, this definition is invariant, and the set
$\set{{\Gamma}_j}_{j=1}^\infty$ may have 
subsequences converging to several 
different limit bicharacteristics, which could be points.
In fact, if ${\Gamma}_j$ is parametrized with respect to the arc
length on intervals $I_j$ such 
that $|I_j| \to 0$, then we find that ${\Gamma}_j$ converges
to a limit curve which is a point.
Observe that if ${\Gamma}_j$ converge to a limit bicharacteristic
${\Gamma}$, then~\eqref{hpcond0} and~\eqref{hpcond} hold for~${\Gamma}_j$.

\begin{exe}\label{circlex}
Let ${\Gamma}_j$ be the curve parametrized by
\begin{equation*}
[0, 1]\ni  t \mapsto {\gamma}_j(t) = (\cos(jt), \sin(jt))/j
\end{equation*}
Since $|{\gamma}'_j(t)| = 1$ the curves are parametrized with
respect to arc length, and we have that ${\Gamma}_j \to (0,0)$ in
$C^0$, but not in $C^\infty$ since $|{\gamma}_j''(t)| = j$. If we 
parametrize ${\Gamma}_j$ with $x = jt \in [0, j]$ we find that
${\Gamma}_j$ converge to $(0,0)$ in $C^\infty$ but not on uniformly
bounded intervals.  
\end{exe}

\begin{exe}\label{rootex}
Let $P$ have real principal symbol $p = w_1^k - a(w')$ in the coordinates
$(w_1,w') \in T^*\br^n\setminus 0$, where $w_1 = a(w') = 0$ at ${\Sigma}_2$ and
$k \ge 2$. This case includes the cases where the operator is
microhyperbolic, then $k = 2$ and $a(w')$ vanishes of exactly second
order at ${\Sigma}_2$. If $k$ is even then in order to have
limit bicharacteristics, it is necessary that $a \not \le 0$ in a
neighborhood of ${\Sigma}_2$, and then ${\Sigma}$ is given locally by
$w_1 = \pm {a(w')}^{1/k}$. We find that~\eqref{hpcond} is
satisfied if $H_{\wt p}^j w_1 = 0$ and $H_{\wt p}^j
\nabla a = \Cal O(|\nabla a|)$ for any $j$ when $w_1 = 0$.
\end{exe}

But we shall also need a condition on the differential of the Hamilton
vector field $H_p$ at the bicharacteristics along a Lagrangean space,
which will give bounds on the 
curvature of the characteristics in these directions. In the
following, a section of Lagrangean spaces $L$
over a bicharacteristic ${\Gamma}$ will be a map 
$$
{\Gamma} \ni w \mapsto L(w) \subset T_w(T^*X)
$$ 
such that $L(w)$ is a Lagrangean subspace in $T_w{\Sigma}$,
$\forall\,w \in \Gamma$, where ${\Sigma} = p^{-1}(0)$. 
If the section $L$ is $C^1$
then it has tangent space $TL \subset T_{L}(T_{\Gamma}(T^*X))$. 
Since the Hamilton vector field $H_{ p}(w)$ is in $T_w\Gamma \subset 
T_w(T^*(X))$, we find the linearization  (or first order jet)
of $H_p$ at $w \in \Gamma$ is in $T_{H_p(w)}(T_{w}(T^*X))$.

\begin{defn}\label{lagrangedef}
For a bicharacteristic ${\Gamma}$ of\/ $p$ we say that a $ C^1 $ section of
Lagrangean spaces $L$ over \/${\Gamma}$ is a section of {\em grazing
Lagrangean spaces\/} of ${\Gamma}$ if $L \subset T{\Sigma}$ and the
linearization (or first order jet) of $H_{p}$ is in $TL$ at $ \Gamma $.
\end{defn}

Observe that since $L(w) \subset T_w{\Sigma}$ is Lagrangean
we find $ dp(w)\restr {L(w)} = 0 $ and
$H_p(w) \in L(w)$ when $w \in {\Gamma}$. The linearization
of $H_p(w)$ is given by the second order Taylor expansion of $p$ at
$w$ and since $L(w)$ is Lagrangean we find that terms in that
expansion that vanish on $L(w)$
have Hamilton field parallel to~$L$. Thus, the condition that the
linearization of $H_p(w)$ is in $ TL(w)$ only depends on the
restriction to $L(w)$ of the second order Taylor expansion of  $p$ at
$w$. Since $L(w) \subset T_w{\Sigma}$, we find that Definition~\ref{lagrangedef} is invariant
under multiplication of $p$ by non-vanishing factors because $p(w)  = dp(w)\restr{L(w)} = 0$.
Thus we can replace $H_p(w)$ by the normalized Hamilton field $H_{\wt p}$ in the definition.

\begin{exe}
Let $L$ be given by ${\Gamma} \ni w \mapsto (w,L(w))$ with coordinates
chosen so that $L(w)= \set{(x,{\xi}) \in \br^{2n}:\ {\xi} = A(w)x}$
where $A(w): \br^n \mapsto \br^n$ is linear. Then we can
parametrize the tangent space of $L(w_0)$ with
\begin{equation*}
 L'(w_0):\ T_{w_0}{\Gamma}\times L(w_0)  \ni ({\delta}w, {\delta}z) \mapsto (w_0,L(w_0),
 {\delta}w, ({\delta}x, A'(w_0)({\delta}w){\delta}x) + {\delta}z)
\end{equation*}
and $L$ is $C^1$ if $A \in
C^1({\Gamma}, \br^n)$.
In these coordinates, the linearization of
$H_p(w_0)$ is given by
\begin{equation*}
 H_p'(w_0):\  T_{H_p(w_0)}(T_{w_0}(T^*X)) \ni ({\delta}x,
 {\delta}{\xi}, {\delta}z, {\delta}{\zeta}) \mapsto
 H_p(w_0)(\delta z, {\delta}{\zeta}) +
 \partial_{x,{\xi}}H_p(w_0)({\delta}x,{\delta}{\xi}) 
\end{equation*}
where $H_p(w_0)(\delta z, {\delta}{\zeta}) = \partial_{\xi}
p(w_0){\delta}z - \partial_{x}p(w_0){\delta}{\zeta} \in L(w_0)$.
\end{exe}

By Definition~\ref{lagrangedef} we find that the linearization of $H_{p}$ 
gives an evolution equation for the
section $L$, see Example~\ref{grazex}.
Choosing a Lagrangean subspace of $T_{w_0}{\Sigma}$
at $w_0 \in {\Gamma}$ then determines $L$ along ${\Gamma}$, so~$L$
must be smooth.
Actually, $L$ is the tangent space at ${\Gamma}$ of a Lagrangean
submanifold of ${\Sigma}$, see~\eqref{newtau}.

\begin{exe}\label{grazex}
Let $p = {\tau} - \left(\w{A(t)x,x} + 2 \w{B(t)x,{\xi}} +
\w{C(t){\xi},{\xi}}\right)/2$ where $(x,{\xi}) \in T^*\br^n$,
$A(t)$, $B(t)$ and $C(t)\in C^\infty$ are real
$n\times n$ matrices, where $A(t)$ and $C(t)$ are symmetric,
and let ${\Gamma} = \set{(t,0,0,{\xi}_0):\ t \in
  I}$. Then $ H_p = \partial_t $ at $ \Gamma $,
$$p^{-1}(0) = \set{{\tau} = \w{A(t)x,x}/2 + \w{B(t)x,{\xi}}
  + \w{C(t){\xi},{\xi}}/2}$$ 
and the linearization of the Hamilton field $H_p$ at $(t,0,0,{\xi}_0)$ is 
\begin{equation}\label{linham}
\partial_t +  \w{A(t)y + B^*(t){\eta}, \partial_{\eta}} - \w{B(t)y +
        C(t){\eta}, \partial_y} \qquad (y,\eta) \in T^*\br^n
\end{equation}
The linearization of the normalized Hamilton field $ H_{\wt p} $ is the same as~\eqref{linham}, 
since $ |H_p| \cong 1 $ modulo quadratic terms in $ (y,\eta) $.
Since $dp = d{\tau}$ at ${\Gamma}$, a $ C^1 $ section of  Lagrangean spaces in~$ T_{\Gamma}\Sigma $ is for example given by 
$$
L(t) = \set{(s,y,0,E(t)y): (s,y) \in \br^{n}}
$$ 
where $ E(t) \in C^1 $ is symmetric with $E(0) = 0$, and by choosing linear symplectic coordinates $ (y,\eta) $ we 
can obtain any such section on this form.
By applying~\eqref{linham} on $ \eta - E(t)y $, which vanishes on~$ L(t) $, we obtain that
$ L(t) $ is a grazing Lagrangean space if
\begin{equation}\label{eveqE}
E'(t) = A(t) + 2 \re B(t)E(t) + E(t)C(t)E(t)
\end{equation}
with $ \re F =(F + F^*)/2 $, see~\eqref{eveq}. Then by uniqueness
$L(t)$ is constant in $t$ if and only if $A(t) 
\equiv 0$, observe that then $A(t) =  \hess p \restr {L(t)}$.
In general, $ \hess p \restr {L(t)} $ is given by the right hand side of~\eqref{eveqE}.
\end{exe}

Observe that we may choose symplectic coordinates $(t,x;{\tau},{\xi})$
so that ${\tau}= p$ and 
the fiber of $L(w)$ is equal to $ \set{(s,y,0, 0): (s,y) \in \br^{n}}$
at $w \in {\Gamma} 
= \set{(t,0;0,{\xi}_0): \ t \in I}$. But it is not clear that we
can do that {\em uniformly} for a family of bicharacteristics $ \set{\Gamma_j}$, for that we need an additional condition.
Now we assume that there exists a grazing Lagrangean space $L_j$ of
${\Gamma}_j$, $\forall\, j$, such that the
normalized Hamilton vector field $H_{\wt p}$ satisfies 
\begin{equation}\label{cond0}
 \left|d H_{\wt p}(w)\restr {L_j(w)} \right| \le C \qquad \text{for
  $w \in {\Gamma}_j$ }\quad \forall\, j
\end{equation}
Since the mapping ${\Gamma}_j \ni w \mapsto L_j(w)$ is determined by
the linearization of $H_{\wt p}$ on $L_j$, thus by $d H_{\wt
  p}(w)\restr {L_j(w)}$, 
condition~\eqref{cond0} implies that 
${\Gamma}_j \ni w \mapsto L_j(w)$ is uniformly in $C^1$, see
Example~\ref{grazex}. Observe that 
condition~\eqref{hpcond} gives~\eqref{cond0} in the direction of $T_w
{\Gamma}_j \subset L_j(w)$. Clearly condition~\eqref{cond0} is 
invariant under changes of symplectic coordinates and multiplications with
non-vanishing real factors.  Now if $ 0 \ne u \in C^\infty$ has values in $\br^n$ and ${\omega}
= u/|u|$ is the  normalization, then
\begin{equation*}
 \partial {\omega} = \partial u/|u| - \w{u, \partial u}u/|u|^3
\end{equation*}
This gives that~\eqref{cond0} is equivalent to
\begin{equation}\label{cond00}
 \left|{\Pi}d\nabla p\restr {L_j} \right| \le C|\nabla p| \qquad \text{on
  ${\Gamma}_j$ }\quad \forall\, j
\end{equation}
where ${\Pi}v = v - \w{v, \nabla p}\nabla p/|\nabla p|^2$ is the orthogonal
projection along 
$\nabla p$. Condition~\eqref{cond00} gives a
uniform bound on the curvature of level surface
$p^{-1}(0)$ in the directions given by $L_j$ over ${\Gamma}_j$. 
Observe that the invariance of condition~\eqref{cond00} can be checked
directly since $d (ap) = a d p$ and 
$d^2 (ap) = a d^2 p + d a d p  = a d^2 p$ on $L_j$
over ${\Gamma}_j$. 

The invariant subprincipal symbol~$p_s$ will be
important for the solvability of the operator.
For the usual Kohn-Nirenberg quantization of
pseudodifferential operators, the subprincipal 
symbol is equal to 
\begin{equation}\label{subprinc}
 p_{s} = p_{m-1} - \frac{1}{2i} \sum_j\partial_{{\xi}_j}\partial_{x_j} p 
\end{equation}
and for the Weyl quantization it is $p_{m-1}$.

Now for the principal symbol we shall denote
\begin{equation}\label{cond1}
  0 < \min_{\Gamma_j}|H_p| = {\kappa}_j \to 0 \qquad j \to \infty
\end{equation}
and for the subprincipal symbol $p_s$ we shall assume the following
condition 
\begin{equation}\label{cond2}
 \min_{\partial \Gamma_j}\int \im p_s |H_p |^{-1}\, ds/ |\log
 {\kappa}_j| \to \infty \qquad j \to \infty
\end{equation}
where the integration is along the natural orientation given by $H_p$
on ~${\Gamma}_j$ starting at some point $w_j 
\in \overset \circ {\Gamma}_j$.  This means that the integral must
have a minimum less or equal to zero in the interior of  $ \Gamma_j $,
and starting the integration at that minimum can only
improve~\eqref{cond2}.  Thus we may assume that the integral
in~\eqref{cond2} is non-negative.  
Observe that if~\eqref{cond2} holds then there must be
a change of sign of $\im p_s$ from $-$ to $+$ on ${\Gamma}_j$, and
\begin{equation}\label{c1}
\max_{\Gamma_j} (-1)^{\pm 1}\im p_s/|H_p| |\log {\kappa}_j| \to \infty
\qquad j \to \infty 
\end{equation} 
for both signs.
Observe that condition~\eqref{cond2} is invariant under symplectic changes of
coordinates and multiplication with elliptic pseudodifferential
operators, thus under conjugation with elliptic Fourier integral
operators. In fact, multiplication only changes the subprincipal symbol 
with the same non-vanishing factors as $|H_p|$ and terms proportional
to $|\nabla p| = |H_p|$. Now by choosing
symplectic coordinates $(t,x;{\tau},{\xi})$ near a given point $w_0 \in
{\Gamma}_j$  so that $p = {\alpha}{\tau}$ near $w_0$ with ${\alpha}
\ne 0$, 
we obtain that $\partial_{x_j}\partial_{{\xi}_j} p =
0$ at~${\Gamma}_j$ and $\partial_t\partial_{\tau} p
= \partial_t {\alpha} = \partial_t|\nabla p|$ at ${\Gamma}_j$ near
$w_0$. Thus, the second term in~\eqref{subprinc} only gives terms in
condition~\eqref{cond2} which are bounded by
\begin{equation} 
\int 
\partial_t|\nabla p|/|\nabla p|\, ds/|\log({\kappa}_j)| \ls |\log(|\nabla
p|)|/|\log({\kappa}_j)| \ls 1 \qquad j \gg 1
\end{equation}
since ${\kappa}_j \le |\nabla p|$ on~${\Gamma}_j$.
Here $a \ls b$ (and $b \gs a$) means that $a \le Cb$ for some $ C > 0$.
Thus we obtain the
following remark.

\begin{rem}\label{subpr}
We may replace the subprincipal symbol $p_s$ by $p_{m-1}$
in~\eqref{cond2}, since the difference is bounded as $j \to \infty$.
\end{rem}

We shall study the microlocal solvability, which is
given by the following definition. Recall that $H^{loc}_{(s)}(X)$ is
the set of distributions that are locally in the $L^2$ Sobolev space
$H_{(s)}(X)$. 

\begin{defn}\label{microsolv}
If $K \subset S^*X$ is a compact set, then we say that $P$ is
microlocally solvable at $K$ if there exists an integer $N$ so that
for every $f \in H^{loc}_{(N)}(X)$ there exists $u \in \Cal D'(X)$ such
that $K \bigcap \wf(Pu-f) = \emptyset$. 
\end{defn}

Observe that solvability at a compact set $M \subset X$ is equivalent
to solvability at $S^*X\restr M$ by~\cite[Theorem 26.4.2]{ho:yellow},
and that solvability at a set implies solvability at a subset. Also,
by Proposition 26.4.4 in~\cite{ho:yellow} the microlocal solvability is
invariant under conjugation by elliptic Fourier integral operators and
multiplication by elliptic pseudodifferential operators.
The following is the main result of the paper.

\begin{thm}\label{mainthm}
Let $P \in {\Psi}^m_{cl}(X)$ have real principal symbol ${\sigma}(P)
= p$, and subprincipal symbol $p_s$. Let
$\set{{\Gamma}_j}_{j=1}^\infty$ be a family of bicharacteristic
intervals of $p$ in $S^*X$ so
that~\eqref{cond0} and~\eqref{cond2} hold. Then $P$ is not 
microlocally solvable at any limit bicharacteristics of~
$\set{{\Gamma}_j}_j$. 
\end{thm}

To prove  Theorem~\ref{mainthm} we shall use the following result. Let
$\mn{u}_{(k)}$ be the $L^2$ Sobolev norm of order $k$ for $u \in
C_0^\infty$ and $P^*$ the $L^2$ adjoint of $P$.

\begin{rem} \label{solvrem}
If $P$ is microlocally solvable at ${\Gamma}\subset S^*X$,
then Lemma 26.4.5 in~\cite{ho:yellow} gives that for any $Y \Subset
X$ such that ${\Gamma} \subset S^*Y$ there exists an integer ${\nu}$ 
and a pseudodifferential operator $A$ so that
$\wf(A) \cap {\Gamma} = \emptyset$ and
\begin{equation}\label{solvest}
 \mn {u}_{(-N)} \le C(\mn{P^*{u}}_{({\nu})} + \mn {u}_{(-N-n)} +
 \mn{Au}_{(0)}) \qquad u \in C_0^\infty(Y)
\end{equation}
where~$N$ is given by Definition~\ref{microsolv}. Since that
definition also holds for larger $N$, we may take ${\nu} = N \ge 0$.
\end{rem}

We shall use Remark~\ref{solvrem} in Section~\ref{pfsect} to prove
Theorem~\ref{mainthm} by constructing approximate local solutions to
$P^* u = 0$. We shall first prepare and get a microlocal normal form
for the adjoint operator, which will be done in
Section~\ref{norm}. Then we shall solve the eikonal equation
in Section~\ref{eik} and the transport equations in Section~\ref{transp}.

\section{Examples}\label{exem}

\begin{exe}
Let $P$ have principal symbol $p = \prod_j p_j$ which is a product of
real symbols~$p_j$ of principal type, such that $p_j = 0$ on $\st$,
$\forall\, j$, and $p_j \ne p_k$ on ${\Sigma}\setminus {\Sigma}_2$ when $j 
\ne k$. We find that if ${\Gamma} \subset p_k^{-1}(0)$ then 
$$
|H_{p}|  = |H_{p_k}| \prod_{j \ne k}|p_j| = |H_{p_k}|q_k
$$
where  $q_k > 0$ on ${\Gamma}\setminus \st$ close to $\st$.
Then $p$ satisfies~\eqref{hpcond} and~\eqref{cond0} for any section of
Lagrangean spaces in $ T p_k^{-1}(0) $ at $ \Gamma \setminus \st$ by the invariance, since $\nabla 
p_k \ne 0$ and $\partial^{\alpha} p_k = \Cal O(1)$, $\forall\, {\alpha}$.  
A bicharacteristic ${\Gamma} \subset {\Sigma}\setminus {\Sigma}_2$ of $p$ is
a bicharacteristic for $p_k$ for some $k$. Then if $p_s$ is the
subprincipal symbol and  
\eqref{cond2} is satisfied for a  
sequence of bicharacteristics of $p_k$ converging in~$C^\infty$
to~$\st$, we obtain that the operator is not solvable at any limit of these
bicharacteristics at~$\st$. 
\end{exe}

\begin{exe}
Assume that $p(x,{\xi})$ is real and vanishes of exactly order $k\ge 2$ at the
involutive submanifold $\st = \set{{\xi}'=0}$, ${\xi}=
({\xi}',{\xi}'') \in \br^m \times \br^{n-m}$, such that the
localization 
$${\eta} \mapsto \sum_{|{\alpha}| = k}
\partial_{{\xi}'}^{\alpha} p(x,0,{\xi}'') {\eta}^{\alpha}
$$ 
is of
principal type when ${\eta} \ne 0$. Then the bicharacteristics of $p\/$
satisfies~\eqref{hpcond} and \eqref{cond0}  
with $L_j = \set{{\xi}=0}$ at any point. In fact, 
$
|\partial_{{\xi}'} p(x,{\xi})|
\cong |{\xi}'|^{k-1}
$
and $\partial_{x,{\xi}''} p(x,{\xi}) = \Cal
O(|{\xi}'|^{k})$ 
so we obtain this since $H_{\wt p} = \partial_{{\xi}'}\wt p \partial_{x'} +
\Cal O(|{\xi}'|)$ and $\partial_x^{\alpha}
\nabla p = \Cal O(|{\xi}'|^{k-1})$
when $|{\xi}'| \ll 1$ and $|{\xi}| \cong 1$.
The operator is not solvable if the
imaginary part of the subprincipal symbol $\im p_s$ changes sign from
$-$ to $+$ along a convergent sequence of  bicharacteristics of the
principal symbol and 
vanishes of at most order $k-2$ at $\st$. In particular we obtain the results
of~\cite{MU1} and~\cite{Men}.
\end{exe}

\begin{exe}
As in Example~\ref{rootex}, let $P$ have real principal symbol $p = w_1^k -
a(w')$ in the coordinates 
$(w_1,w') \in T^*\br^n\setminus 0$, where $w_1 = a(w') = 0$ at ${\Sigma}_2$. We
find that~\eqref{cond0} is 
satisfied if $d w_1\restr{L_j} = 0$ and $d
\nabla a\restr{L_j} = \Cal O(|\nabla a|)$ when $w_1 = 0$. 
\end{exe}

\begin{exe}\label{quadratic}
Let $Q$ be a real and hyperbolic quadratic form on $T^*\br^n$. Then
by~\cite[Theorem 1.4.6]{ho:cauchy} we have  the following normal form
$ 
Q_1(x,{\xi}) + Q_2(y,{\eta})
$ 
where the symplectic coordinates $ (x,y; \xi, \eta) \in T^*(\br^m \times \br^r )$,
\begin{equation*}
 Q_1(x,{\xi}) = \sum_{j=1}^{k} {\mu}_j (x_j^2 + {\xi}_j^2) +
 \sum_{j=k+1}^{m} {\xi}_j^2  \qquad {\mu}_j > 0\quad \forall\, j
\end{equation*}
is positive semidefinite and $Q_2(y,{\eta})$ is either $-{\eta}_1^2$,
${\mu}_0y_1{\eta}_1$ with $ \mu_0 \ne 0 $ or $2{\eta}_1{\eta}_2 - y_2^2$.
To simplify the notation, we will assume ${\mu}_j = 1$, $\forall\,
j$. 

Now the Hamilton vector field $ H_Q(x,y; \xi, \eta)  = H_{Q_1}(x,\xi) +  H_{Q_2}(y,\eta) $ 
so the flow of $H_Q$ is a direct sum of  the flows of $ H_{Q_1} $ and $H_{Q_2}$.
For $H_{Q_1}$ it is a direct sum of the flows on the circles in $(x_j, {\xi}_j)$ of 
radius $r_j \to 0$, $j \le k$, and the flows on  $x_j$ lines, $k < j \le m$. The
circles can only converge in $C^\infty$ to the origin if the radii~$r_j$ go to zero fast
enough, see Example~\ref{circlex}. The
possible limits are then points or lines in the $x_{k+1},\dots, x_m$ space. 

In the case $Q_2(y,{\eta}) = -{\eta}_1^2$ we find that the limit
bicharacteristics are given by 
\begin{equation}\label{limbic}
 y_1 = {\lambda}t, \ x_j = {\lambda}_jt + a_j, \ k < j \le m
\end{equation}
where ${\lambda}^2 = \sum_{j= k+1}^{m} {\lambda}_j^2$. We can only
find a Lagrangean space satisfying~\eqref{cond0} if ${\mu}_j = 0$,
$\forall\, j$, since such a space cannot be contained in the subspace  $\set{x_j
  = {\xi}_j = 0}$. Then the grazing Lagrangean space can be chosen as~
$\set{(x,y;0,0)}$ at every point of the bicharacteristic.
Theorem~\ref{mainthm} then gives that the operator with homogeneous
principal symbol 
equal to $Q(x,y;{\xi},{\eta})$   is
not solvable when the imaginary part of the 
subprincipal symbol changes sign on the lines on $\st$ given
by~\eqref{limbic}, which also follows from the results in~\cite{Men}. 

If $Q_2(y,{\eta}) = y_1{\eta}_1$, then $Q^{-1}(0) = \set{ y_1{\eta}_1
  = -Q_1(x,{\xi}) = -{\lambda}^2}$ where $|H_{Q_1}| \cong
{\lambda}$. The Hamilton vector field  
of $Q_2$ is given by $ H_{Q_2} = y_1\partial_{y_1} -\eta_1
\partial_{\eta_1} $ which has flow $t \mapsto 
(y_1e^{t};{\eta}_1e^{-t})$ with starting point $(y_1,{\eta}_1)$. Here $y_1{\eta}_1
= -{\lambda}^2 < 0$ which gives $|y_1| + |{\eta}_1| \gs |y_1\eta_1|^{1/2} \gs
{\lambda}$. We find that $|H_Q| \cong L =|y_1| + |{\eta}_1|$, so in
order for the bicharacteristic to converge in $C^\infty$, we find
from~\eqref{hpcond} that $(|y_1| + |{\eta}_1|)/L^k = L^{1-k}\ls 1$,
$\forall \, k$. But then  
$|y_1| + |{\eta}_1| \not \to 0$ so the bicharacteristics do not
converge in $C^\infty$ to a limit bicharacteristic at~$\st$. 
Observe that a hyperbolic operator with
Hessian of the principal symbol equal to $Q(x,y;{\xi},{\eta})$ 
at~$\st\bigcap S^*\br^n$ 
is {\em effectively hyperbolic} and is solvable with any lower
order terms, see the Notes to Chapter 23 in \cite{ho:yellow}.

Finally, when $Q_2(y,{\eta}) = 2{\eta}_1{\eta}_2 - y_2^2$, then the
characteristics in the $(y,{\eta})$ variables are
$\set{{\eta}_1{\eta}_2 = (y_2^2 + {\lambda}^2)/2}$
where ${\lambda}^2 = Q_1$. We
find $H_{Q_2} = 2{\eta}_2 \partial_{y_1} + 2{\eta}_1\partial_{y_2} +
2y_2\partial_{{\eta}_2} $ so ${\eta}_1$ is constant on the
orbits. Note that $y_2^{2} + {\lambda}^2 = 2 {\eta}_1{\eta}_2 \le
{\eta}_1^2 + 
{\eta}_2^2$. Thus when $|{\eta}_1| \gs |{\eta}_2|$ we find that
$|{\eta}_1| \gs |y_2| + {\lambda}$ on~${\Sigma}$. Now $| H_{ \wt Q_2}
\nabla \wt Q_2| \gs  \eta_1^{-1}$ so in order 
for~\eqref{hpcond}  
to hold we must have $|{\eta}_1| \gs 1$, so the characteristics will
in that case not converge in $C^\infty$ to any limit bicharacteristic at~$\st$.
When $|{\eta}_2| \gg |{\eta}_1|$ we find that
$|{\eta}_2| \gg |y_2| + {\lambda}$ on~${\Sigma}$. 
A straighforward computation shows that applying $ H_{ \wt Q} $ twice  on $
\nabla \wt Q $ gives a factor $ \eta_1/\eta_2^3 $. Thus, 
by~\eqref{hpcond}  we find that the
bicharacteristics converge in $C^\infty$ to the $y_1$ lines on~$\st$ only if 
$|{\eta}_1| \ls |{\eta}_2|^3$, which implies that $|y_2| + {\lambda} 
\ls {\eta}_2^2$. Then $H_{\wt Q} = \partial_{y_{1}} + \mathcal
O(|{\eta}_2|)$ so the bicharacteristics can only converge to $y_1$ lines
at ${\eta}_2 \to 0$.  
An example is the case when $y_2 = {\eta}_1 = {\lambda} = 0$ and ${\eta}_2 \to 0$, then $ H_{Q_2} = 2{\eta}_2 \partial_{y_1} $ and the grazing Lagrangean space can be chosen as 
$\set{(x,y_1,0;0,0,{\eta}_2)}$ at every point of the bicharacteristics.
Thus Theorem~\ref{mainthm}
gives that the operator with homogeneous principal symbol
equal to $Q(x,y;{\xi},{\eta})$ when $|{\xi}|^2 + |{\eta}|^2 = 1$ is
not solvable when the imaginary part of the 
subprincipal symbol changes sign on $y_1$ lines at $\st$.
\end{exe}

\section{The normal form}\label{norm}

First we shall put the adjoint operator $P^*$ on a
normal form uniformly and microlocally near the bicharacteristics
${\Gamma}_j$ converging to ${\Gamma}$. 
This will present some difficulties since we only have conditions at the
bicharacteristics. By the invariance, we may multiply with an elliptic
operator so that the order of $P^*$ is $m=1$ and $P^*$ has the symbol
expansion $p + p_0 + 
\dots$, where $p$ is the principal symbol. As in Remark~\ref{subpr} we may assume
that $p_0$ is the subprincipal 
symbol. Observe that for the adjoint the signs in~\eqref{cond2} are
reversed and it changes to  
\begin{equation}\label{newcond2} 
\max_{\partial \Gamma_j}\int \im p_0 |H_p |^{-1}\, ds/ |\log
 {\kappa}_j| \to -\infty \qquad j \to \infty
\end{equation}
where  ${\kappa}_j$ given by~\eqref{cond1}.
By changing $w_j$ to the maximum of the integral
in~\eqref{newcond2}  only improves the estimate so we may assume that
\begin{equation}\label{c3}
 \int \im p_0/|H_p|\, ds \le 0 \qquad \text{on ${\Gamma}_j$}
\end{equation}
with equality at $w_j \in {\Gamma}_j$.
Since $\nabla \im p_0$ and $\nabla H_p$ are bounded on $S^*X$ and $|H_p| \ge
{\kappa}_j$ on~${\Gamma}_j$, we find that
$ |H_p |$ and
$\im p_0/|H_p|$ only change with a fixed factor and a bounded term on an interval of
length $\ll {\kappa}_j$ on $ \Gamma_j $. Therefore, we find
that integrating $\im p_0/|H_p|$ over such intervals only gives bounded terms.
Thus, we find  from~\eqref{c1} that 
\begin{equation}\label{c4}
 |{\Gamma}_j| \gg {\kappa}_j 
\end{equation}
and that condition~\eqref{cond2} holds on intervals of length
$\cong {\kappa}_j$ at the endpoints of ${\Gamma}_j$.

Now we choose
\begin{equation}\label{ldef}
1 \ls {\lambda}_j =
{\kappa}_j^{-1/{\varepsilon}} \iff {\kappa}_j =
{\lambda}_j^{-{\varepsilon}}
\end{equation}
for some ${\varepsilon} > 0$ to be determined later.
Then we may replace$|\log {\kappa}_j|$  with
$\log {\lambda}_j$ in \eqref{cond2}--\eqref{c1}. 
By choosing a subsequence and renumbering, we may assume
by~\eqref{cond2} that 
\begin{equation}\label{mcond2}
 \min_{\partial \Gamma_j} \int \im p_0/|H_p| \, ds \le  -j\log {\lambda}_j
\end{equation}
and that this holds on intervals of length $\cong {\kappa}_j$ at the
endpoints of ${\Gamma}_j$. Next, we introduce the normalized principal
and subprincipal symbols 
\begin{equation*}
 \wt p = p/|H_p|  \qquad \text{and} \qquad p_s = p_{0}/|H_p|
\end{equation*}
then we have that $H_{\wt p}\restr {\Gamma_j} \in C^\infty$ uniformly,
$|H_{\wt p}| = 1$ on ${\Gamma}_j$ and $dH_{\wt p}\restr {L_j}$ is
uniformly bounded at ${\Gamma}_j$. We find that
condition~\eqref{mcond2} becomes
\begin{equation}\label{rcond2}
 \min_{\partial \Gamma_j}\int \im p_s \, ds  \le  -j\log {\lambda}_j
\end{equation}
Observe that because of
condition~\eqref{c1} we have that 
$\partial {\Gamma}_j$ has two components since $\im p_s$ has opposite
sign there, so ${\Gamma}_j$ is a uniformly embedded curve.

In the following we shall consider a
fixed bicharacteristic ${\Gamma}_j$ and suppress the index $j$, so
that ${\Gamma} = {\Gamma}_j$, $L = L_j$ and ${\kappa} = {\kappa}_j =
{\lambda}^{-{\varepsilon}}$. Observe that the preparation will be
uniform in~$j$. Now $H_{\wt p} \in C^\infty$ uniformly on~${\Gamma}$
but not in a neighborhood. By~\eqref{cond0} we may define the first jet
of~$\wt p$ at~${\Gamma}$ uniformly. 
Since ${\Gamma} \in C^\infty$ uniformly, we can choose local
coordinates uniformly so that ${\Gamma} = \set{(t,0)):\ t \in I \subset
  \br}$ locally. In fact, take a local parametrization
${\gamma}(t)$ of ${\Gamma}$ with respect to the arclength and choose the
orthogonal space $M$ 
to the tangent vector of ${\Gamma}$ at a point $w_0$ with respect to some
local Riemannean metric. Then $\br \times
M \ni (t,w) \mapsto {\gamma}(t) + w $ is uniformly bounded in $C^\infty$
with uniformly bounded inverse near $(t_0,0)$ giving local coordinates
near~${\Gamma}$. We may then complete~$ t $ to a symplectic coordinate
system uniformly. We can define the first order Taylor term of $\wt p$ at
${\Gamma}$ by 
\begin{equation}\label{qdef}
{\varrho}(t,w) = \partial_{w} \wt 
p(t,0)\cdot w \qquad w = (x,{\tau},{\xi})
\end{equation}
which is uniformly bounded.
This can be done locally, and by using a uniformly bounded partition of 
unity we obtain this in a fixed neighborhood of
${\Gamma}$. Going back to the original coordinates, we find that
${\varrho} \in C^\infty$ uniformly near~${\Gamma}$ and $\wt p - 
{\varrho} = \Cal O(d^2)$, but the
error is not uniformly bounded. Here $d$ is the homogeneous distance
to ${\Gamma}$, i.e., the distance with respect
to the homogeneous metric $dt^2 + |dx|^2 + (d{\tau}^2 +
|d{\xi}|^2)/\w{({\tau},{\xi})}^2$. 
But by condition~\eqref{cond0} we find
that the second order derivatives of $\wt p$ along the Lagrangean space $L$
at ${\Gamma}$ are uniformly bounded. 
We shall use homogeneous coordinates, i.e., local coordinates which are normalized 
with respect to the homogeneous metric. 

By completing ${\tau} = {\varrho}$ in~\eqref{qdef} to a uniformly
bounded homogeneous symplectic  
coordinate system near ${\Gamma}$ and conjugating with the
corresponding uniformly bounded Fourier integral 
operator we may assume that
\begin{equation}\label{gammaco}
{\Gamma} = \set{(t,0;0, {\xi}_0): \ t \in I} 
\end{equation} 
for  $|{\xi}_0| = 1$, some
fixed interval $I \ni 0$, and $\wt p \cong {\tau}$ modulo second order terms
at ${\Gamma}$. The second order terms are not uniformly bounded, but
$d\nabla  \wt p \restr {L}$ is uniformly bounded at ${\Gamma}$
by~\eqref{cond0}. Since $d\wt p = d \tau$ on~${\Gamma}$ we 
find that $H_{\wt p}\restr {\Gamma} = D_t$ and we may obtain 
that $L =  \set{(t,x;0,0)}$ at any given point at~${\Gamma}$ by 
choosing linear symplectic coordinates~$ (x,\xi) $.

Let 
\begin{equation}\label{wtpdef}
 q(t,w) =  |\nabla p(t,w)| \ge 
{\kappa} = {\lambda}^{-{\varepsilon}} \qquad \text{ at  
${\Gamma}$}
\end{equation}
and extended so that $q$ is homogeneous of degree $0$, which is the norm of
the  homogeneous gradient of~$p$. Recall that $ \lambda $ is a parameter that depends on the 
bicharacteristic~$ \Gamma $.
We shall change coordinates so that $(t,w) = 0$ corresponds to the point
$(0,0; 0, {\xi}_0) \in {\Gamma}$ given 
by~\eqref{gammaco}, and then localize in neighborhoods depending on~$ \lambda $. 
We have $|\nabla \wt p| \equiv 1 $ at~${\Gamma}$,
higher derivatives are not uniformly bounded but can be handled by
the using the metric 
\begin{equation*}
 g_{\varepsilon} = (dt^2 +
|dw|^2){\lambda}^{2{\varepsilon}}
\end{equation*}
and the symbol classes $f \in S(m,g_{\varepsilon})$ defined by  
$\partial^{\alpha} f = \Cal O(m{\lambda}^{|{\alpha}|{\varepsilon}})$,
$\forall\, {\alpha}$.

\begin{prop}\label{wtpest}
We have that $q$ is a weight for  $g_{\varepsilon}$, $q  \in S(q,
g_{\varepsilon})$ and $\wt p(t,w) \in
S({\lambda}^{1-{\varepsilon}}, 
g_{\varepsilon} )$ when $|w| 
 \le c{\lambda}^{-{\varepsilon}}$  for some $c > 0$  when $t \in I$.
\end{prop}

This gives  $p = q\wt p \in 
S(q{\lambda}^{1-{\varepsilon}}, g_{\varepsilon})$ when $|w| 
\le c{\lambda}^{-{\varepsilon}}$. 
Observe that $b \in S^{\mu}_{1-{\varepsilon}, {\varepsilon}}$ if
and only if 
$b \in S({\lambda}^{\mu}, g_{\varepsilon})$ in homogeneous coordinates
when $|{\xi}| \gs  
{\lambda}  \gs 1$. In fact, in homogeneous coordinates $z $ this
means that $\partial_z^{\alpha} b = \Cal O(|{\xi}|^{{\mu} +
  |{\alpha}|{\varepsilon}})$ when $|z| \cong 1$.

\begin{proof}
We shall use the previously chosen coordinates $(t,w)$ so that ${\Gamma} =
\set{(t,0): \ t \in I}$. Now $\partial^2p = \Cal
O(1)$,  $q \ge {\lambda^{-\varepsilon}}$ at ${\Gamma}$ by~\eqref{wtpdef} and
\begin{equation}\label{qder}
 \partial q = \nabla p\cdot (\partial\nabla p)/q 
 \qquad \text{when $q \ne 0$} 
\end{equation}
which is uniformly bounded.
We find that $q(s,w) \cong q(t,0)$ 
when $|s-t| + |w| \le c {\lambda^{-\varepsilon}}$ for small enough $c > 0$,
so $q$ is a 
weight for $g_{\varepsilon}$ there. This gives that 
$|p(t,w)| \ls q(t,w){\lambda^{-\varepsilon}}$, 
$|\nabla p(t,w)| = q(t,w)$ and $|\partial^\alpha p | \ls 1 \ls q{\lambda^{\varepsilon}}$ 
for $ |\alpha| \ge 2 $,
which gives $p \in 
S(q{\lambda^{-\varepsilon}}, g_{\varepsilon})$ when $|w| 
\le c{\lambda^{-\varepsilon}}$.  

We find from~\eqref{qder} that $ \partial  q = 
\Cal O(q{\lambda^{\varepsilon}})$  when $|w| 
\le c{\lambda^{-\varepsilon}}$,
since $\nabla p \in S(q, g_{\varepsilon})$ in this domain. By
induction over the order of differentiation we obtain from~\eqref{qder} that $q \in 
S(q,g_{\varepsilon})$ 
when $|w| \le c{\lambda^{-\varepsilon}}$, which gives the result. 
\end{proof}

Next, we put $Q(t,w) = {\lambda^{\varepsilon}}\wt
p(t{\lambda^{-\varepsilon}},w{\lambda^{-\varepsilon}})$ when  $t \in 
I_{\varepsilon}  = \set{t{\lambda^{\varepsilon}}:\ t \in I}$. Then by
Proposition~\ref{wtpest} we find that $Q \in C^\infty$ 
uniformly when $|w| \ls 1$ and  $t \in I_{\varepsilon}$,
$\partial_{\tau} Q \equiv 1$  and $|\partial_{t,x,{\xi}} Q| 
\equiv 0$ when $w = 0$ and $t \in I_{\varepsilon}$. Thus we find
$|\partial_{\tau} Q| \ne 0$ for 
$|w| \ls 1$ and $t \in I_{\varepsilon}$. 
By using Taylor's formula at~$ \Gamma $ we
can write $Q(t,x;{\tau},{\xi}) = {\tau} + h(t,x;{\tau},{\xi})$
when $|w| \ls 1$ and  $t \in I_{\varepsilon}$,
where $h = |\nabla h| = 0$ at $w = 0$.
By using the Malgrange preparation theorem, we find
\begin{equation*}
{\tau} = a(t,w)({\tau} + h(t,w)) + s(t,x,{\xi})
 \qquad |w| \ls 1 \quad t \in I_{\varepsilon}
\end{equation*}
where $a$ and $s \in C^\infty$ uniformly, $a \ne 0$, and on
${\Gamma}$ we have $a = 1$ and $s =
|\nabla s| = 0$. In fact, this can be done locally in $t$ and by a
uniform partition of unity for $t \in I_\varepsilon$. This gives 
\begin{equation}\label{sdef}
a(t,w)Q(t,w) = 
{\tau} -  s(t,x,{\xi})
 \qquad |w| \ls 1 \quad t \in I_{\varepsilon}
\end{equation}
In the original coordinates, we find that 
\begin{equation*}
 {\lambda}^{\varepsilon}\wt p(t,w) =
 a^{-1}(t {\lambda}^{\varepsilon},w {\lambda}^{\varepsilon})({\tau}
 {\lambda}^{\varepsilon} 
 - s(t {\lambda}^{\varepsilon},x {\lambda}^{\varepsilon},{\xi}
 {\lambda}^{\varepsilon})) 
\end{equation*}
and thus
\begin{equation}\label{locprep}
 \wt p(t,w) = b(t,w)({\tau} - r(t,x,{\xi}))  \qquad |w| \ls
 {\lambda}^{-{\varepsilon}} \quad  t  \in I  
\end{equation}
where $0 \ne b \in S(1, g_{\varepsilon})$, $r(t,x,{\xi}) =
{\lambda^{-\varepsilon}}s(t {\lambda}^{\varepsilon},x
{\lambda}^{\varepsilon},{\xi}{\lambda}^{\varepsilon}) \in
S({\lambda^{-\varepsilon}}, 
g_{\varepsilon})$ when $  |w| \ls
{\lambda}^{-{\varepsilon}} $, and $ t  \in I $, $b = 1$ and $r = |\nabla r| = 0$
on ${\Gamma}$.  
By condition~\eqref{cond0} we also 
find that 
\begin{equation}\label{cond000}
 \left|d \nabla r\restr L \right| \le C \qquad \text{when $w = 0$ and
   $t \in I$}
\end{equation}
Recall that $\wt p = p/q$, where $q \in S(q, g_{\varepsilon})$ when $ |w| \ls \lambda^{-\varepsilon} $. 

In the following, we shall denote by ${\Gamma}$ the rays in $T^*X$
that goes through the bicharacteristic.
By homogeneity we obtain from~\eqref{locprep} that
\begin{equation*}
 b^{-1}q^{-1} p(t,x;{\tau},{\xi}) = {\tau} - r(t,x,{\xi}) 
\end{equation*}
where $b^{-1} \in S^0_{1-{\varepsilon},{\varepsilon}}$, $q^{-1} \in
S^{\varepsilon}_{1-{\varepsilon},{\varepsilon}}$ and ${\tau} - r \in
S^{1-{\varepsilon}}_{1-{\varepsilon},{\varepsilon}}$ 
when $ |\xi| \gs \lambda $ and the homogeneous distance $d$ to
${\Gamma}$ is less than $c|{\xi}|^{-\varepsilon}$ for some $c > 0$. 

Next we take a cut-off
function ${\chi} \in S^0_{1-{\varepsilon},{\varepsilon}}$ supported
where $d \ls 
|{\xi}|^{-\varepsilon}$ and $  |\xi | \gs \lambda  $ so that $b \ge c_0 > 0$ in $\supp \chi$ and ${\chi} = 1$ when 
$  |\xi | \ge c \lambda  $ and $d \le c 
|{\xi}|^{-\varepsilon}$ for any chosen $c > 0$. We let $B =
{\chi}b^{-1}q^{-1} \in \se {\varepsilon}$ and compose
the pseudodifferential operator $B$ with $P^*$. Since $P^* 
\in {\Psi}^1_{1,0}$ we obtain an asymptotic expansion of $BP^*$ in
$S^{1 + \varepsilon
  -j(1-{\varepsilon})}_{1-{\varepsilon},{\varepsilon}}$ for $j= 0$, 1,
$2, \dots$ when $d \ls
|{\xi}|^{-\varepsilon}$. But actually the symbol is in a better class.
The principal symbol is 
\begin{equation*}
({\tau} - r(t,x,{\xi})){\chi} \in
S^{1-{\varepsilon}}_{1-{\varepsilon},{\varepsilon}}  \qquad 
  d \ls|{\xi}| ^{-{\varepsilon}} 
\end{equation*}
and the calculus gives that the subprincipal symbol is equal to
\begin{equation}\label{sub}
\frac{i}{2}H_p({\chi}b^{-1}q^{-1}) + {\chi}b^{-1}q^{-1}p_0
\end{equation}
where  $p_0$ is the subprincipal symbol of $P^*$. As before, we shall
use homogeneous coordinates when $|{\xi}| \gs {\lambda}$.
Then Proposition~\ref{wtpest} gives $p = q\wt p \in
S(q{\lambda}^{1-{\varepsilon}}, g_{\varepsilon})$ 
and since ${\chi}b^{-1}q^{-1} \in S(q^{-1}, g_{\varepsilon})$ we find
that~\eqref{sub}  
is in $S({\lambda}^{\varepsilon}, g_{\varepsilon})$ when $d \ls
{\lambda}^{-{\varepsilon}}$ and $|{\xi}| \gs {\lambda}$. The value
of $H_p$ at  ${\Gamma}$ is equal to $q\partial_t$ so the value
of~\eqref{sub} is equal to
\begin{equation}
\frac{1}{2i}\partial_t q/q + p_0/q = \frac{D_t
|\nabla_h p| }{2|\nabla_h p|} + \frac{p_0}{|\nabla_h p|} \qquad \text{ when $  |\xi| \gs \lambda  $ at
${\Gamma}$}
\end{equation}
where $|\nabla_h p| = \sqrt{|\partial_x p|^2/|{\xi}|^2 + |\partial_{\xi}
  p|^2}$ is the homogeneous gradient, and the error of this approximation
is bounded by ${\lambda}^{2{\varepsilon}}$ times 
the homogeneous distance $d$ to ${\Gamma}$. In fact, since $\partial p
\in S(q ,g_{\varepsilon}) $ 
by Proposition~\ref{wtpest}, we 
have $H_p q^{-1} \in S({\lambda}^{{\varepsilon}}, g_{\varepsilon})$
and $p_0/q \in S(q^{-1}, g_{\varepsilon})$. Observe that $p_0/|\nabla_h
p|$ is equal to the normalized subprincipal symbol of $P^*$ on $S^*X$.
This preparation can only be done when $  |\xi| \gs \lambda  $
and the homogeneous distance  to
${\Gamma}$ is less than $c|{\xi}|^{-\varepsilon}$,
but we have to estimate the error terms.

\begin{defn}
For $ \varepsilon < 1/2 $ and $R \in {\Psi}^{\mu}_{{\varrho},{\delta}}$,  where
${\varrho} >{\varepsilon}$ and ${\delta} < 1 - {\varepsilon}$, we say that $T^*X \ni
(x_0,{\xi}_0) \notin \wf_{\varepsilon} (R)$ if the 
symbol of $R$ is $\Cal O(|{\xi}|^{-N})$, $\forall\, N$, when
the homogeneous distance to the ray $\set{(x_0, {\varrho}{\xi}_0):\
  {\varrho} \in \br_+}$ is less than $c|{\xi}|^{-{\varepsilon}}$ for 
some $c > 0$.
\end{defn}

By the calculus, this means that there exists $A\in
{\Psi}^{0}_{1-{\varepsilon},{\varepsilon}}$ so that $A \ge c > 0$ in a
neighborhood of the ray such that $AR \in {\Psi}^{-N}$ for any $N$.
This neighborhood is in fact the points with fixed bounded homogeneous distance
with respect to the metric $g_{\varepsilon}$ when ${\lambda} \cong |{\xi}|$.
It also follows from the calculus that this definition is invariant
under composition with classical pseudodifferential operators and conjugation
with elliptic homogeneous Fourier integral operators 
since the conjugated symbol is given by an asymptotic expansion by the
conditions on ${\varrho}$ and ${\delta}$. 
We also have that $\wf_{\varepsilon}(R) \subset \wf (R)$ when  $R \in
{\Psi}^{\mu}_{{\varrho},{\delta}}$.

Now we can use the Malgrange division theorem in
order to make the lower order terms independent on ${\tau}$ when $d \ls
|\xi|^{-{\varepsilon}} \ls \lambda^{-{\varepsilon}} $, starting with  
the subprincipal symbol  $\wt p_0 \in
S^{{\varepsilon}}_{1-{\varepsilon},{\varepsilon}}$ of $BP^*$ given
by~\eqref{sub}. Then
rescaling as before so that $Q_0(t,w) = 
{\lambda^{-\varepsilon}}\wt p_0(t{\lambda^{-\varepsilon}},
w{\lambda^{-\varepsilon}})$ we obtain that  
\begin{equation*}
 Q_0(t,w) = \wt c(t,w)({\tau}
 -s(t,x,{\xi})) + \wt q_0(t,x,{\xi}) \qquad |w| \ls 1 \quad t \in
 I_{\varepsilon}
\end{equation*}
where $s$ is given by~\eqref{sdef}, and $\wt c$ and $\wt q_0$ are
uniformly in $C^\infty$. This can be 
done locally and by a partition of unity for $t\in I_\varepsilon$. 
We find in the original coordinates that 
\begin{equation}\label{subpreploc}
\wt p_0(t,w) = c(t,w) (  {\tau} - r(t,x,{\xi}))
+ q_0(t,x,{\xi}) \qquad d \ls {\lambda}^{-{\varepsilon}}  \quad t \in I 
\end{equation}
where
$q_0(t,w)  = {\lambda^{\varepsilon}}\wt q_0(t{\lambda^{\varepsilon}},
w{\lambda^{\varepsilon}}) \in S(\lambda^\varepsilon, g_\varepsilon)$ and $c(t,w)  =
 {\lambda}^{2\varepsilon}\wt c(t{\lambda^{\varepsilon}},
 w{\lambda^{\varepsilon}})  \in S(\lambda^{2\varepsilon}, g_\varepsilon)$. 
By using a uniform partition of unity in
${\Psi}^0_{1-{\varepsilon},{\varepsilon}}$, we obtain~\eqref{subpreploc}
uniformly
when $  |\xi| \gs \lambda  $ and the homogeneous distance to
${\Gamma}$ is  $\ls |{\xi}|^{-\varepsilon}$ with $c \in 
S^{2{\varepsilon}-1}_{1-{\varepsilon},{\varepsilon}}$ by homogeneity, $q_0 \in
S^{\varepsilon}_{1-{\varepsilon},{\varepsilon}} $ 
and  $q_0 = \wt p_0$ at ${\Gamma}$. 
Now the composition of the operators having symbols~
$c$ and~${\tau}-r$ gives error terms that are in
$S^{3{\varepsilon}-1}_{1-{\varepsilon},{\varepsilon}}$ when $  |\xi| \gs \lambda  $. 
Thus if ${\varepsilon} < 1/3$ then by multiplication with an
pseudodifferential operator with symbol $1-c$ we can make the
subprincipal symbol independent of~${\tau}$. By iterating this
procedure we can successively make any lower order terms
independent of ${\tau}$ when $  |\xi| \gs \lambda  $ and 
the homogeneous distance $d$ to
${\Gamma}$ is less than $c|{\xi}|^{-\varepsilon}$. 
By applying the cut-off function ${\chi}$ we
obtain the following result.  

\begin{prop} \label{prepprop}
By conjugating with an elliptic
homogeneous Fourier integral operator we may obtain that\/
${\Gamma}$ is given by~\eqref{gammaco}. If\/  $0 < {\varepsilon} < 1/3$
then for any $ c > 0 $ we may multiply with an homogeneous elliptic operator $B \in 
{\Psi}^{{\varepsilon}}_{1-{\varepsilon},{\varepsilon}}$
to obtain that $B P^* = Q + R \in
{\Psi}^{1-{\varepsilon}}_{1-{\varepsilon},{\varepsilon}}$ where
${\Gamma} \bigcap \wf_{\varepsilon} (R) = \emptyset$ and the
symbol of $Q$ is equal to  
\begin{equation}\label{prepsymb}
 {\tau} - r(t,x,{\xi}) + q_0(t,x,{\xi}) +  r_0(t,x,{\xi}) \qquad
 \text{when $d \le c|{\xi}| ^{-{\varepsilon}}$, $  |\xi| \ge c\lambda  $  and $t \in I$}
\end{equation}
Here $r \in
S^{1-{\varepsilon}}_{1-{\varepsilon},{\varepsilon}}$, $q_0 \in 
S^{{\varepsilon}}_{1-{\varepsilon},{\varepsilon}}$ and  $r_0 \in
S^{3{\varepsilon} - 
  1}_{1-{\varepsilon},{\varepsilon}}$, $r = |\nabla r| = 0$ on
${\Gamma}$, and $q_0$ is equal to
\begin{equation}\label{symb}
\frac{D_t |\nabla_h p(t,0)|}{2|\nabla_h p(t,0)|} +
\frac{{p_0}(t,0)}{|\nabla_h p(t,0)|}  + \Cal
O({\lambda}^{2{\varepsilon}}d) \qquad 
 \text{when $d \le c|{\xi}| ^{-{\varepsilon}}$, $  |\xi| \ge c\lambda  $  and $t \in I$}
\end{equation}
where $d$ is the homogeneous distance 
to ${\Gamma}$ and  $|\nabla_h p| = \sqrt{|\partial_x p|^2/|{\xi}|^2 +
|\partial_{\xi} p|^2}$ is the homogeneous gradient. This preparation is uniform with respect 
to~ $ \lambda \ge 1$.
\end{prop}

Observe that the integration of the term $D_t |\nabla_h
p(t,0)|/2 |\nabla_h p(t,0)|$ in~\eqref{symb} will give terms that are
$\Cal O(\log (|\nabla_h p(t,0)|)) = \Cal O(|\log ({\lambda})| + 1)$, which do
not affect condition~\eqref{cond2}. We find from
Proposition~\ref{prepprop} that $\wt p(t,x;{\tau},{\xi}) \cong
{\tau} - r(t,x,{\xi})$ modulo terms vanishing of third order at~${\Gamma}$
since $ r $ vanishes of second order at~$ \Gamma $.

Recall that $L$ is a smooth section of Lagrangean spaces 
$L(w) \subset T_w{\Sigma} \subset T_w(T^*\br^{n})$, $w \in {\Gamma}$, such that
the linearization of the Hamilton vector field $H_p$ is in $TL$ at
${\Gamma}$. 
By Proposition~\ref{prepprop} we may assume that ${\Gamma}
= \set{(t,0;0,{\xi}_0):\ t \in I}$, $0 \in I$,
and $\wt p(t,x;{\tau},{\xi}) = {\tau} - r(t,x,{\xi})$
modulo terms vanishing of third order at ${\Gamma}$. Then we may
parametrize $L(t) = L(w)$ where $w  
= (t,0,{\xi}_0)$ for $t \in I$. Now since $T^*\br^n$ is a linear space, we may
identify the fiber of $T_w(T^*\br^n)$ with $T^*\br^n$. Since $L(w) \subset
T_w\Sigma$ and $w \in {\Gamma}$  
we find that ${\tau} = 0$ in $L(w)$. Since $L(w)$ is Lagrangean, we 
find that $t$ lines are in $L(w)$. By choosing symplectic coordinates in
$(x,{\xi})$ we obtain that $L(0) =  \set{(s,y;0,0):\
  {(s,y)} \in \br^{n}}$, then by condition~\eqref{cond0} we find that 
$\partial_x^2 r(0,0,{\xi}_0)$ is uniformly bounded. 
Since ${\tau} = 0$ on $L(t)$ and $L(t)$ is Lagrangean 
we find for small $t$ by continuity that
\begin{equation}\label{Lex}
 L(t) = \set{(s,y;0,A(t)y):\  {(s,y)} \in
  \br^{n}} 
\end{equation}
where $A(t)$ is real, continuous and symmetric for $t \in I$ and $A(0) =
0$. Since the linearization of the Hamilton vector field $H_p$ at
${\Gamma}$ is tangent to $L$, we find that $L$ is
parallel under the flow of that linearization. 
Since $L(t)$ is Lagrangean it is
only the restriction of the second order jet of $r(t,w)$ to $L(t)$ that 
determines the evolution of $t \mapsto L(t)$. For~\eqref{Lex} this
restriction is given by the second order Taylor expansion of
\begin{equation*}
 R(t,x) = r(t,x,{\xi}_0 + A(t)x)
\end{equation*}
thus $\partial_x^2R(t,0)$ is uniformly bounded by condition~\eqref{cond0}.
The linearized Hamilton vector field is
\begin{multline*}
\partial_t +  \w{\partial_x^2R(t,0)x,\partial_{\xi}} \\= \partial_t +
 \w{\left(\partial_x^2r(t,0,{\xi}_0) + 2 \re (\partial_x\partial_{\xi}r(t,0,{\xi}_0)A)   +
 A\partial_{\xi}^2r(t,0,{\xi}_0)A\right)x,\partial_{\xi}}  
\end{multline*}
where $\re B = (B + {}B^{t})/2 $ is the symmetric part of $B$. Applying this on
${\xi} - A(t) x$, which vanishes identically on $L(t)$ for $t \in I$, we obtain that
\begin{equation*}
 -A'(t) + \partial_x^2r(t,0,{\xi}_0) + 2 \re
 (\partial_x\partial_{\xi}r(t,0,{\xi}_0)A(t)) + 
 A(t)\partial_{\xi}^2r(t,0,{\xi}_0)A(t) =0
\end{equation*}
which gives the evolution of $L(t)$. The equation
\begin{equation}\label{eveq}
 A'(t) = \partial_x^2r(t,0,{\xi}_0) + 2 \re
 (\partial_x\partial_{\xi}r(t,0,{\xi}_0)A(t)) + 
 A(t)\partial_{\xi}^2r(t,0,{\xi}_0)A(t) \qquad A(0)= 0
\end{equation}
is locally uniquely solvable and the right-hand side is uniformly
bounded as long as $A$ is bounded. Observe that  by uniqueness, $A(t) \equiv 0$
if and only if $\partial_x^2r(t,0,{\xi}_0) \equiv 0$, $\forall\,t$
(see also Example~\ref{grazex}).
But since~\eqref{eveq} is non-linear, the solution
could become unbounded if $\partial_x^2r \ne 0$ and $\partial_{\xi}^2r \ne 0$
so that $\mn {A(s)} \to \infty$ as $s \to t_1 \in
I$. This means that the angle between $L(t) = \set{(s,y;0,A(t)y):\ (s,y) \in
  \br^{n}}$ and the vertical space $\set{(s,0;0,{\eta}): (s,{\eta}) \in
  \br^{n}}$ goes to zero, but this is just a coordinate
singularity.

In general, since we identify the fiber of $T_w(T^*\br^n)$ with $T^*\br^n$ we may
define $R(t,x,{\xi})$ for each~$t$ so that
\begin{equation}\label{defR}
R(t,x,{\xi}) = r(t,x,{\xi}_0 + {\xi}) \quad \text{when $(0,x;0,{\xi}) \in L(t)$} 
\end{equation}
Then $R = r$ on $L$ and we find  that 
\begin{equation}\label{newtau}
 {\tau} - \w{R(t)z,z}/2 \in C^\infty 
\end{equation}
if $z= (x,{\xi})$ and $R(t) = \partial^2_{z} R(t,0,0)\restr L(t)$. 
Observe that we find from~\eqref{cond0} that~\eqref{newtau} 
is uniformly in $ C^\infty $ in~$ z $.
We find that $R(0) = \partial_x^2r(t,0, {\xi}_0)$ and in general
$R(t)$ is given by the right hand side of~\eqref{eveq}.
Now we
can complete $t$ and~\eqref{newtau} to a uniform homogeneous symplectic
coordinates system so that 
$
 {\Gamma} = \set{(t,0,{\xi}_0):\ t \in I}
$ and $
L(0) = \set{(s,y;0,0):\ {(s,y)} \in
  \br^{n}}.
$
In fact, we may let $x$ and ${\xi}$ have the same values
when $t = 0$ and clearly $H_{\tau}$ is is not changed on~${\Gamma}$
since then $z=0$. 
This is a uniformly bounded linear symplectic transformation in $(x,{\xi})$
which is uniformly $ C^1 $ in $ t $. It is given by a 
uniformly bounded elliptic Fourier integral operator $ F(t) $ which has uniformly
bounded $ t $ derivative.  We will call this type of Fourier
integral operator a {\em $C^1$ section of Fourier integral operators}. 
This will give uniformly bounded terms when we conjugate a
first order differential operator in $t$ with $ F(t) $, for example
the normal form which has symbol given by~\eqref{prepsymb}.
For $ t $ close to $ 0 $ the section $ F(t) $ is given by multiplication with $ e^{\w {A(t)x,x}/2} $,
where $ A(t) $ solves~\eqref{eveq}. For general $ t $ we can put $ F(t) $ on this form after a linear
symplectic transformation in ~$(x,{\xi})$.
Observe that $ F(t) $ is continuous on local $ L^2 $ Sobolev spaces in $ x $, 
uniformly in ~$ t $, since it is continuous with respect to the norm $ \mn{(1 + |x|^2 + |D|^2)^k u} $, $ \forall\ k $.
In fact, it suffices to check this for the generators of the group of Fourier integral operators corresponding
to linear symplectic transformations, which are given by the partial Fourier transforms, linear transformations in ~$ x $ 
and multiplication with $ e^{i\w{Ax,x}} $ where~$ A $ is real and symmetric.

We find in the new coordinates
that $p = {\tau} - r_1$, where $r_1(t,x,{\xi})$ is independent
of~${\tau}$ and satisfies $ \partial^2_{z}r_1(t,0,0)\restr L(t)
\equiv 0$. This follows since 
$$
 p(t,x;{\tau},{\xi}) = {\tau} - \w{R(t)z,z}/2 - r_1(t,x,{\tau},{\xi})
$$ 
where $\partial_{\tau} r_1 = -
\set{t, r_1} = -\set{t,r} \equiv 0$ is invariant under changes of symplectic
coordinates. Similarly we find that the lower order terms
$p_j(t,x,{\xi})$ are independent of~${\tau}$ for~$j \le 0$. Since the evolution of
$L$ is determined by the second 
order derivatives of the principal symbol along $L$ by
Example~\ref{grazex}, we find that $L(t) \equiv \set{(t,x;0,0):\ {(t,x)} \in
  \br^{n}}$. Changing notation so that $r = r_1$ and
$p(t,x;{\tau},{\xi}) = {\tau} - r(t,x,{\xi})$ we obtain the following result.

\begin{prop}\label{symplprop}
By conjugation with a uniformly bounded $ C^1 $ section of elliptic 
Fourier integral operators, corresponding to  linear symplectic transformations in~$ (x,\xi) $,
we may assume that the symplectic coordinates
are chosen so that the grazing Lagrangean space $L(w) \equiv \set
{(t,x,0,0):\ (t,x) \in \br^{n}}$, $\forall\, w \in {\Gamma}$, which
implies that $\partial_x^2 r(t,0) \equiv 0$.
\end{prop}

We shall apply the adjoint $P^*$ of the
operator on the form in Proposition~\ref{prepprop} on approximate
solutions on the form 
\begin{equation}\label{udef}
 u_{\lambda}(t,x) = \exp(i{\lambda}(\w{x,{\xi}_0} + {\omega}(t,x)))
 \sum_{j=0}^{M} {\varphi}_j (t, x){\lambda}^{-j{\varrho}} 
\end{equation}
where $ |\xi_0|  = 1 $ and the phase function ${\omega}(t,\cdot) \in
S({\lambda}^{-7{\varepsilon}}, g_{3\varepsilon})$ is real valued such
that  $\partial_x^2 {\omega}(t,0) \equiv 0$ and
${\varphi}_j(t,x) \in S(1,g_{\delta})$ has support where $|x| 
\ls {\lambda}^{-{\delta}}$. 
Here ${\delta}$, ${\varepsilon}$ and  
${\varrho}$ are positive constants to be determined later. 
The phase function ${\omega}(t,x)$ will be constructed
in Section~\ref{eik}, see Proposition~\ref{omegalem}.
Observe that we have assumed
that ${\varepsilon} < 1/3$ in Proposition~\ref{prepprop}, but we shall
impose further restrictions on~${\varepsilon}$ later on.
We shall assume that  ${\varepsilon} + {\delta} < 1$, then if
$p(t,x,{\xi}) \in {\Psi}^{1-{\varepsilon}}_{1-{\varepsilon},{\varepsilon}}$  
we obtain the formal expansion (see~\cite[Chapter VI, Theorem 3.1]{T2})
\begin{multline}\label{trevesexp}
  p(t,x,D_x)  (\exp(i{\lambda}(\w{x,{\xi}_0} + {\omega}(t,x))){\varphi}(t,x))
  \\ \sim \exp(i{\lambda}(\w{x,{\xi}_0} +{\omega}(t,x))) \sum_{{\alpha}} 
   \partial_{{\xi}}^{\alpha }
  p(t,x,{\lambda}({\xi}_0 + \partial_x{\omega}(t,x)))\Cal 
  R_{\alpha}({\omega},{\lambda},D){\varphi}(t,x)/{\alpha}!
\end{multline}
where $\Cal
R_{\alpha}({\omega},{\lambda},D){\varphi}(t,x) =  
D_y^{\alpha}(\exp(i{\lambda} \wt
{\omega}(t,x,y)){\varphi}(t,y))\restr{y=x}$ 
with
$$
\wt {\omega}(t,x,y) = {\omega}(t,y) - {\omega}(t,x) +
(x-y)\partial_x {\omega}(t,x)
$$
Observe that if $ |\partial_x\omega| \ll 1 $ then this only involves the values of 
$ p(t,x,\xi) $ where $ |\xi| \ge c\lambda $ for some $c > 0$.
Using this expansion we find that 
\begin{multline}\label{exp}
 P^*(t,x,D)  (\exp(i{\lambda}(\w{x,{\xi}_0} +
 {\omega}(t,x))){\varphi}(t,x)) \\  
\sim \exp(i{\lambda}(\w{x,{\xi}_0} + {\omega}(t,x)))\big({\lambda}\partial_t
{\omega}(t,x) - r(t,x,{\lambda}({\xi}_0 + 
 \partial_x{\omega})){\varphi}(t,x) \\ + D_t {\varphi}(t,x) -
 \sum_{j}\partial_{{\xi}_j}r(t,x,{\lambda}({\xi}_0 + 
 \partial_x{\omega}))D_{x_j}{\varphi}(t,x)  + q_0(t,x,{\lambda}({\xi}_0 +
\partial_x{\omega})){\varphi}(t,x)
\\ +  \sum_{jk}\partial_{{\xi}_j}\partial_{{\xi}_k}r(t,x,{\lambda}({\xi}_0 +
 \partial_x{\omega}))( D_{x_j}D_{x_k}{\varphi}(t,x)
 + i{\lambda}{\varphi}(t,x) D_{x_j}D_{x_k}{\omega}(t,x))/2  +\dots\big)  
\end{multline}
which gives an expansion in
$S({\lambda}^{1- {\varepsilon} - j(1- {\delta}- {\varepsilon})}, 
g_{{\delta}})$, 
$j \ge 0$, if ${\delta} +
{\varepsilon} < 1$ and ${\varepsilon} \le 1/4$. 
In fact, since $|{\xi}| \cong {\lambda}$ every
${\xi}$ derivative on terms in
$S^{1-{\varepsilon}}_{1-{\varepsilon},{\varepsilon}}$ gives
a factor that is $\Cal O({\lambda}^{{\varepsilon}-1})$ and every $x$ derivative
of ${\varphi}$ gives a factor that is $\Cal
O({\lambda}^{{\delta}})$. A factor ${\lambda} D_x^{\alpha}{\omega}$
requires $|{\alpha}| $ number of ${\xi}$ derivatives of a term in
the expansion of 
$P^*$, which gives a factor that is $\Cal
O({\lambda}^{(2 - |{\alpha}|)(1-4{\varepsilon})})$. 
Similarly, the expansion coming from terms in $P^*$ that have symbols in
$S^{\varepsilon}_{1-{\varepsilon},{\varepsilon}}$ gives an expansion
in $S^{\varepsilon - j(1 - {\delta} -
  {\varepsilon})}_{1-{\varepsilon},{\varepsilon}}$, $j \ge 0$.
Thus, if ${\delta} + {\varepsilon} < 2/3$ and ${\varepsilon} \le 1/4$
then the terms in the expansion have negative  
powers of ${\lambda}$ except the terms in~\eqref{exp}, and for the
last ones we find that 
\begin{equation}\label{exp0}
 \sum_{jk}\partial_{{\xi}_j}\partial_{{\xi}_k}r(t,x,{\lambda}({\xi}_0
 + \partial_x{\omega}))(D_{x_j}D_{x_k}{\varphi} + i{\lambda}{\varphi}D_{x_j}D_{x_k}{\omega})
 = \Cal 
 O({\lambda}^{2{\delta} +
   {\varepsilon} -1} + {\lambda}^{3{\varepsilon} - {\delta}}) 
\end{equation}
In fact,  $\partial_{{\xi}_j}\partial_{{\xi}_k}r(t,x,{\lambda}({\xi}_0
+ \partial_x{\omega})) = \Cal O({\lambda}^{{\varepsilon}-1})$ and
 $D_{x_j}D_{x_k}{\omega} = \Cal O({\lambda}^{2{\varepsilon}}d)$
when ${\varphi} \ne 0$, since we have $D_{x_j}D_{x_k}{\omega} = 0$ when $x =
0$, and $d  = \Cal O({\lambda}^{-{\delta}})$ in $\supp {\varphi}$. 

The error terms in~\eqref{exp0} are of equal size if $2{\delta} +
{\varepsilon} -1 = 3{\varepsilon} -
{\delta}$, i.e., ${\delta} = (1+2
{\varepsilon})/3$. We then obtain 
$3 {\varepsilon} - {\delta} 
= (7{\varepsilon} -1)/3 < 0$ if
${\varepsilon} < 1/7$. Observe that in this case 
$1 - {\delta} - {\varepsilon} = (2 - 5{\varepsilon})/3$ and
${\delta}+ {\varepsilon} < 2/3$ since ${\varepsilon} < 1/5$.
We also have that $1 - 4{\varepsilon} > 1 - {\delta} - {\varepsilon}$
since ${\delta} > 3{\varepsilon}$.
Thus we obtain the following result.

\begin{prop}\label{expprop}
Assume that  ${\omega}(t,\cdot) \in
S({\lambda}^{-7{\varepsilon}}, g_{3\varepsilon})$ is real valued and
$\partial_x^2 {\omega}(t,0) \equiv 0$,
${\varphi}_j(t,x) \in S(1,g_{\delta})$ has support where $|x| 
\ls {\lambda}^{-{\delta}}$, for ${\delta}$, ${\varepsilon} > 0$.  If
${\delta} = (1+ 2{\varepsilon})/3$ and ${\varepsilon} < 1/7$,
then~\eqref{exp} has an expansion in $S({\lambda}^{1-
  {\varepsilon} - j(2 - 5{\varepsilon})/3}, 
g_{{\delta}})$, $j \ge 0$, and
is equal to 
\begin{multline}\label{exp1}
 \exp(-i{\lambda}(\w{x,{\xi}_0} + 
 {\omega}(t,x))) P^*(t,x,D) 
   (\exp(i{\lambda}(\w{x,{\xi}_0} +
     {\omega}(t,x))){\varphi}(t,x)) \\ \sim 
({\lambda}\partial_t {\omega}(t,x) -
 r(t,x, {\lambda}({\xi}_0 +
 \partial_x{\omega}))){\varphi}(t,x) \\ +  D_t {\varphi}(t,x) -
 \sum_{j}\partial_{{\xi}_j}r(t,x, {\lambda}({\xi}_0 +
 \partial_x{\omega})) D_{x_j}{\varphi}(t,x) + q_0(t,x, {\lambda}({\xi}_0 +
 \partial_x{\omega})){\varphi}(t,x)
\end{multline}
modulu terms that are $\Cal O({\lambda}^{(7{\varepsilon} -1)/3})$.
\end{prop}
  
In Section~\ref{transp} we shall choose ${\varepsilon} = {\varrho} = 1/10$
which gives ${\delta} = 2/5$, $1 - 4{\varepsilon} = 3/5$, $1 - {\delta} -{\varepsilon} =
1/2$ and $(7{\varepsilon} -1)/3 = -1/10$.

\section{The eikonal equation}\label{eik}

The first term in the expansion~\eqref{exp1} is the eikonal equation 
\begin{equation}\label{eic}
\partial_t {\omega} - s(t,x,{\xi}_0 + \partial_x {\omega}) = 0 \qquad
{\omega}(0,x) = 0 
\end{equation}
where $s(t,x,\xi) = \lambda     ^{-1} r(t,x,\lambda \xi)$. 
This we can solve by using the Hamilton-Jacobi equations:
\begin{equation}\label{hamjac}
\left\{
\begin{aligned} 
&\partial_t x = -\partial_{\xi}s(t,x,{\xi}_0 +{\xi})\\
&\partial_t{\xi} = \partial_{x}s(t,x,{\xi}_0 +{\xi})
\end{aligned}
\right.
\end{equation}
with $(x(0), {\xi}(0)) = (x,0)$, and letting $\partial_x {\omega}
= {\xi}$ and $\partial_t {\omega} = s(t,x,\xi_0 + \partial_x {\omega})$
with $ {\omega}(0,x) = 0 $.
Since $s = \nabla s = 0$ on ${\Gamma}$ we find that $\partial_t x =
\partial_t {\xi} = 0$ when
$x= {\xi} = 0$, so  by uniqueness  $\partial_t {\omega}(t,0) \equiv \partial_x {\omega}(t,0) \equiv 0$.

We shall solve the Hamilton-Jacobi equations by scaling. Recall that
$s(t,x,{\xi}) \in S({\lambda^{-\varepsilon}},g_{\varepsilon})$  for
some chosen $0 < {\varepsilon} < 1/3$ in 
homogeneous coordinates by Proposition~\ref{prepprop}. By
Proposition~\ref{symplprop} we may assume 
that $L(t)\equiv \set  
{(t,x,0,0)}$, $\forall\, t$, and $\partial_x^2 s = 0$ on~${\Gamma}$.
Since $s$, $\partial s $ and $\partial_x^2 s$
vanish on ${\Gamma}$, Taylor's formula gives
\begin{equation}\label{eiceq}
 \partial_{\xi}s(t,x,{\xi}_0 + {\xi}) =
 \partial_x\partial_{\xi}s(t,0,{\xi}_0)x + 
 \partial_{\xi}^2s(t,0,{\xi}_0){\xi} + \w{{\varrho}_1(t,x,{\xi})w,w} 
\end{equation}
where $w = (x,{\xi})$, $\partial_x\partial_{\xi}s(t,0,\xi_0) = \Cal O(1)$
by~\eqref{cond0}, $\partial_{\xi}^2s(t,0,{\xi}_0) = \Cal O(
{\lambda}^{\varepsilon})$ and 
${\varrho}_1 \in S({\lambda}^{2{\varepsilon}}, g_{\varepsilon})$, since $ \wt p(t,x;\tau, \xi) \cong \tau - r(t,x,\xi) $ 
modulo terms vanishing of third order at~$ \Gamma $.
Similarly, we find
\begin{equation}\label{eiceq0}
 \partial_x s(t,x,{\xi}_0 +{\xi}) = \partial_x\partial_{\xi}s(t,0,{\xi}_0){\xi} +
 \w{{\varrho}_2(t,x,{\xi})w,w} 
\end{equation}
where ${\varrho}_2 \in S({\lambda}^{2\varepsilon},
g_{\varepsilon})$.  

Now we put $(x, {\xi}) = (y {\lambda}^{-3\varepsilon}, {\eta}
{\lambda}^{-4\varepsilon})$. Then by
using~\eqref{eiceq} and~\eqref{eiceq0} we find that ~\eqref{hamjac}
transforms into
\begin{equation}\label{scaleic}
\left\{
\begin{aligned} 
&\partial_t y = -B(t)y  - C(t) {\eta} + {\sigma}_1(t,z)\\
&\partial_t {\eta} = B(t){\eta} +  {\sigma}_2(t,z)
\end{aligned}
\right.
\end{equation}
where $z= (y,{\eta})$, $B(t) = \partial_x\partial_{\xi}s(t,0,{\xi}_0)$ and
$C(t) = {\lambda^{-\varepsilon}} 
\partial_{\xi}^2s(t,0,{\xi}_0)$ are uniformly bounded, and 
\begin{equation}
\left\{ 
\begin{aligned}
&{\sigma}_1(t,z) = 
{\lambda}^{-3\varepsilon}\w{{\varrho}_1(t,y{\lambda}^{-3\varepsilon},
{\eta} {\lambda}^{-4\varepsilon})(y,
{\lambda}^{-{\varepsilon}}{\eta}), (y, {\lambda}^{-{\varepsilon}}{\eta})}\\
&{\sigma}_2(t,z) =
- {\lambda}^{-2\varepsilon}\w{{\varrho}_2(t,y{\lambda}^{-3\varepsilon},
{\eta}{\lambda}^{-4\varepsilon})(y, {\lambda}^{-{\varepsilon}}{\eta}),
(y, {\lambda}^{-{\varepsilon}}{\eta})}
\end{aligned}
\right.
\end{equation}
are uniformly
bounded in $C^\infty$ and vanish of second order in $z$. Then
~\eqref{scaleic} has a uniformly 
bounded $C^\infty$ solution if $z(0)$ is uniformly bounded. This means
that if $x(0)= \Cal O({\lambda}^{-3\varepsilon})$ and ${\xi}(0) = 0$
then we find  $x=
\Cal O({\lambda}^{-3\varepsilon})$ and $\partial_x {\omega} = {\xi} = \Cal
O({\lambda}^{-4\varepsilon})$ for any ~$t\in I$. 
The scaling gives that
$\partial_x^{\alpha}\partial_x{\omega} =  \Cal
O({\lambda}^{(-4+3|{\alpha}|){\varepsilon}})$ when $|x| \ls
{\lambda}^{-3{\varepsilon}}$. Since
${\omega}(t,0) \equiv 0$ we obtain that ${\omega} = \Cal
O({\lambda}^{-7\varepsilon})$ when $|x| \ls
{\lambda}^{-3{\varepsilon}}$, thus  ${\omega}(t,\cdot) \in
S({\lambda}^{-7{\varepsilon}}, 
g_{3{\varepsilon}})$.

Now by differentiating~\eqref{eic} twice we find that
\begin{equation*}
 \partial_t \partial_x^2 {\omega}(t,0) = 
 2\re \big(\partial_{\xi}\partial_xs(t,0,{\xi}_0)\partial_x^2
 {\omega}(t,0)\big) +   \partial_{x}^2 
 {\omega}(t,0) \partial_{\xi}^2s(t,0,{\xi}_0) \partial_{x}^2
 {\omega}(t,0) 
\end{equation*}
because $\partial_x {\omega}(t,0) =
\partial_{\xi}s(t,0,{\xi}_0) =  \partial_x^2s(t,0,{\xi}_0) = 0$.
Since $\partial_x^2{\omega}(0,x)\equiv 0$ 
we find by uniqueness that $\partial_x^2  
{\omega}(t,0) \equiv 0$. This gives that $\partial_t {\omega} = \Cal
O({\lambda}^{-7{\varepsilon}})$ when $|x| \ls
{\lambda}^{-3{\varepsilon}}$, and we obtain the following result.

\begin{prop}\label{omegalem}
Let\/ $0 < {\varepsilon} < 1/3$, and assume that
Propositions~\ref{prepprop} and ~\ref{symplprop} hold.
Then there exists a real ${\omega}(t,\cdot) \in
S({\lambda}^{-7{\varepsilon}}, 
g_{3{\varepsilon}})$ satisfying $\partial_t {\omega} =
\lambda^{-1} r(t,x,\lambda({\xi}_0 + \partial_x {\omega}))$ when $|x| \ls
{\lambda}^{-3{\varepsilon}}$  and $t \in I$ such that ${\omega}(t,0)
\equiv 0 $ and $\partial_x^2{\omega}(t,0) \equiv 0$.
We  find that the values of\/
$(t,x;  {\lambda}\partial_{t}{\omega}(t,x), {\lambda}({\xi}_0 +
\partial_{x}{\omega}(t,x)))$ have homogeneous distance to the rays through\/ ${\Gamma}$
which is $\ls {\lambda}^{-3\varepsilon} $ when $|x| \ls
{\lambda}^{-3{\varepsilon}}$ and $t \in I$.
\end{prop}

\section{The transport equations}\label{transp}

The next term in~\eqref{exp1} is the transport equation, which is
equal to
\begin{equation}\label{trans}
D_p {\varphi}_0  +
 {q_0}{\varphi}_0 = 0 \qquad \text{at ${\Gamma}$}
\end{equation}
where $D_p = D_t - \sum_{j}\partial_{{\xi}_j}r(t,x,\lambda({\xi}_0 +\partial_x
{\omega(t,x)})) D_{x_j}  = D_t$ when $x = 0$ and 
\begin{equation}\label{q0def}
 q_0(t) = D_t
  |\nabla p(t,0,{\xi}_0)|/2|\nabla p(t,0,{\xi}_0)| + {p_0}(t,0,{\xi}_0)/|\nabla
  p(t,0,{\xi}_0)| = \Cal O({\lambda}^{\varepsilon}) 
\end{equation}
modulo $\Cal O({\lambda}^{2{\varepsilon}}|x|)$ when $|x| \ls
{\lambda}^{-{\delta}}$ by~\eqref{symb}. Here $ {\omega}(t,x) $ is given
by Proposition~\ref{omegalem}.

\begin{lem}\label{translemma}
We have that  
\begin{equation*}
 D_p = D_t + \sum_{j}^{} \w{a_j(t), x} D_{x_j} + R(t,x,D_x) 
\end{equation*}
where $\br^{n-1} \ni a_j(t) = \Cal O(1)$ and $R(t,x,D_x)$ is a first order
differential operator in $x$ with 
coefficients that are $\Cal O({\lambda}^{3{\varepsilon}}|x|^2)$ when $|x|
\ls {\lambda}^{-3{\varepsilon}}$.
\end{lem}

\begin{proof}
As before, we shall put $ s(t,x,\xi) = 
\lambda^{-1} r(t,x,\lambda \xi) \in  S( {\lambda^{-\varepsilon}},g_{\varepsilon})$.
Since $\partial_x^2{\omega}(t,0) \equiv 0$ by Proposition~\ref{omegalem}
we have from Taylor's formula that $a_j(t) = 
 - \partial_x \partial_{{\xi}_j}s(t,0,{\xi}_0)$ which is uniformly bounded
 by~\eqref{cond000} and Proposition~\ref{symplprop}.
The coefficients of the error term $R$ are given by the second order
$x$ derivatives of the coefficients of $D_p$ which are 
\begin{equation*}
 \partial_x^2\partial_{\xi}s + 2\re \big(
 \partial_x\partial_{\xi}^2s\partial_x^2 {\omega}\big) +\partial_x^2 {\omega}
 \partial_{\xi}^3s\partial_x^2 {\omega} + \partial_{\xi}^2 s
 \partial_x^3 {\omega} = \Cal O({\lambda}^{3{\varepsilon}})
\end{equation*}
when $|x| \ls {\lambda}^{-3{\varepsilon}} \ll \lambda^{-\varepsilon}$ by
Propositions~\ref{prepprop} and~\ref{omegalem}, which proves the result.
\end{proof}

We obtain new variables $y$ in $\br^{n-1}$ by solving
\begin{equation*}
 \partial_t y_j = \w{a_j(t), y} \qquad y_j(0) = x_j \qquad \forall\, j
\end{equation*}
Then $D_t + \sum_{j}^{} \w{a_j(t), x} D_{x_j}$ is changed into
$D_t$ but $D_{x_j} = D_{y_j}$ is unchanged, and we will for simplicity keep the notation $(t,x)$. 
The change of variables is uniformly bounded since $a_j = \Cal O(1)$, so
it preserves the neighborhoods $|x| \ls {\lambda}^{-{\nu}}$ and
symbol classes $S({\lambda}^{\mu}, g_{\nu})$, $\forall\, {\mu},\, {\nu}$.
We shall then solve the approximate transport equation
\begin{equation}\label{mtrans}
 D_t{\varphi}_0 + {q_0}(t){\varphi}_0 = 0
\end{equation}
where ${\varphi}_0(0,x) \in S(1, g_{\delta})$ is supported where $|x| \ls
{\lambda}^{-{\delta}}$ and $q_0(t)$ is given by~\eqref{q0def}.  
If ${\lambda}^{-{\delta}} \ll {\lambda}^{-3{\varepsilon}}$ then by
Lemma~\ref{translemma} the approximation errors will be in 
$S({\lambda}^{3{\varepsilon}-{\delta}}, g_{{\delta}})$, so we
will assume ${\delta} > 3{\varepsilon}$. In fact, since
$\partial_x$ maps $ S(1, g_{\delta})$ into $S({\lambda}^{{\delta}},
g_{\delta})$ and $|x| \ls {\lambda}^{-{\delta}}$, we find $R(t,x,D_x){\varphi}_0
\in  S({\lambda}^{3{\varepsilon} - {\delta}}, g_{\delta})$. 
If we put ${\delta} = 4 {\varepsilon}$ then the approximation errors
in the transport equation will be $\Cal O({\lambda}^{-{\varepsilon}})$.

If we choose the initial data
${\varphi}_0(0,x) = {\phi}_0(x)= {\varphi}({\lambda}^{{\delta}}x)$,
where ${\varphi} 
\in C^\infty_0$ satisfies ${\varphi}(0) = 1$, we obtain the solution
\begin{equation}
 {\varphi}_0(t,x) =
 {\phi}_0(x) \exp(- i B(t)) 
\end{equation}
where $B' = q_0$ and B(0) = 0. 
By condition~\eqref{c3} we find that 
$ |{\varphi}_0(t,x)| \le  |{\varphi}({\lambda}^{{\delta}}x)|$, and $|x| \ls
{\lambda}^{-{\delta}}$ in $\supp {\varphi}_0$, which also holds in 
the original $x$ coordinates.
Observe that $D_{x}^{\alpha}{\varphi}_0 = \Cal 
O({\lambda}^{\delta|{\alpha}|})$, $\forall\, {\alpha}$, and we have
from the transport 
equation that $D_t {\varphi}_0 = - q_0{\varphi}_0 = \Cal
O({\lambda}^{{\varepsilon}})$ by~\eqref{q0def}. Since $D_t^k q_0 = \Cal
O({\lambda}^{\varepsilon(k+1)})$ 
by Proposition~\ref{prepprop}, we find
by induction that ${\varphi}_0 \in S(1, g_{\delta})$. 

After solving the eikonal equation and the approximate transport
equation, we find from Proposition~\ref{expprop} that the terms in the
expansion~\eqref{exp1} 
are $\Cal O({\lambda}^{(7{\varepsilon} -1)/3}) + \Cal
O({\lambda}^{-{\varepsilon}})$, if  ${\varepsilon} < 1/7$
and ${\delta} = (1+2{\varepsilon})/3 = 4{\varepsilon}$,  
and all the terms
contain the factor $ \exp(- i B(t))$. We take ${\varepsilon}  =
1/10$, ${\delta} = 2/5$ which gives
$(7{\varepsilon} -1)/3 = -{\varepsilon} = -1/10$. Then the expansion
in Proposition~\ref{expprop} is in
multiples of ~${\lambda}^{-1/2}$, but since the terms 
of~\eqref{exp1} are $\Cal O({\lambda}^{-1/10})$
we will take ${\varrho} = 1/10$.

Thus the approximate transport equation for ${\varphi}_1$ is
\begin{equation}
 D_t{\varphi}_1 + {q_0}(t){\varphi}_1 = {\lambda}^{1/10}
R_1  \exp(-i B(t))  \qquad \text{at ${\Gamma}$}
\end{equation}
where $R_1$  is
uniformly bounded in the symbol class $S({\lambda}^{-1/10},
g_{2/5})$ and  supported where $|x| \ls  {\lambda}^{-{2/5}}$. In
fact, $R_1$ contains both the error terms 
from the transport equation~\eqref{trans} for~ ${\varphi}_0$ and the
terms that are $\Cal O({\lambda}^{-1/10})$ in~\eqref{exp1}.
By putting 
$$
{\varphi}_1(t,x) =
\exp(-iB(t)){\phi}_1(t,x)
$$ 
the transport equation reduces to solving 
\begin{equation}\label{ltrans}
 D_t {\phi}_1 = {\lambda}^{1/10}
R_1  
\end{equation}
with initial values ${\phi}_1(0,x) = 0$. Then $ {\phi}_1 \in  S(1,
g_{2/5})$ will have support where $|x| \ls
{\lambda}^{-{2/5}}$. 
 
Similarly, the general term in the expansion is
${\varphi}_k{\lambda}^{-k/10}$ where 
${\varphi}_k$ will solve the approximate transport equation
\begin{equation}\label{gentrans}
D_t{\varphi}_k + q_0(t){\varphi}_k = {\lambda}^{k/10} R_k
\exp(i B(t))  \qquad k \ge 1
\end{equation}
with $R_k$   is
uniformly bounded in the symbol class $S({\lambda}^{-k/10},
g_{2/5})$ and  is  supported where $|x| \ls  {\lambda}^{-{2/5}}$.
In fact, $R_k$ contains the error terms from the transport
equation~\eqref{trans} for ~${\varphi}_{k-1}$ 
and also the terms that are $\Cal O({\lambda}^{-k/10})$ in \eqref{exp1}.
Taking
 ${\varphi}_k =  \exp(-i B(t)) {\phi}_k$ we obtain the equation 
\begin{equation}\label{transk}
 D_t{\phi}_k = {\lambda}^{k/10} R_k \in S(1,
 g_{2/5})
\end{equation}
with initial values ${\phi}_k(0,x) = 0$,
which can be solved with ${\phi}_k \in 
 S(1, g_{2/5})$ uniformly having support where  $|x| \ls
{\lambda}^{-{2/5}}$. Proceeding we obtain an solution modulo $\Cal
O({\lambda}^{-N/10})$ for any~ $N$.

\begin{prop}\label{transprop}
Choosing ${\delta} = 2/5$, ${\varepsilon} = 1/10$ and ${\varrho} =
1/10$ we can solve the transport equations~\eqref{trans}
and~\eqref{gentrans} with ${\varphi}_k \in S(1,g_{2/5})$ 
uniformly having support where  $|x| \ls {\lambda}^{-{2/5}}$
and $ |t| \ls 1 $, $\forall\, k$,
such that ${\varphi}_0(0,0) = 1$ and ${\varphi}_k(0,x)
\equiv 0$, $k \ge 1$. 
\end{prop}

Now, we get localization in $x$ from the initial values and the
transport equation. To get
localization in $t$ we use that $\im B(t) \le 0$. Then
we find that $\re (-i B) \le 0$ with equality at $t =
0$. Near $\partial {\Gamma}$ we may assume that
$\re (-iB(t))  \ll -\log {\lambda}$ in an interval of length
$\Cal O({\lambda^{-\varepsilon}}) = \Cal O({\lambda}^{-1/10})$
by~\eqref{c4}, \eqref{ldef} and~\eqref{rcond2}. 
Thus by applying a cut-off function 
$$
{\chi}(t) \in S(1, {\lambda}^{1/5}dt^2) \subset  S(1, g_{2/5})
$$ 
such that ${\chi}(0) = 1$ and ${\chi}'(t)$ is supported
where~\eqref{mcond2}  holds, i.e., where 
${\varphi}_k = \Cal O({\lambda}^{-N})$, $\forall\, k$, we obtain a solution modulo $\Cal
O({\lambda}^{-N})$ for any $N$. In fact, if $u_{\lambda}$ is defined
by~\eqref{udef} and $Q$ by Proposition~\ref{prepprop} then
\begin{equation*}
 Q {\chi} u_{\lambda} = {\chi}Q u_{\lambda} + [Q, {\chi}]u_{\lambda}
\end{equation*}
where $[Q, {\chi}] = D_t {\chi}$ is supported
where $u_{\lambda} = \Cal O({\lambda}^{-N})$ which gives terms that are $\Cal
O({\lambda}^{-N})$, $\forall\, N$. Thus, by solving the eikonal
equation~\eqref{eic} for ${\omega}$ and the transport
equations~\eqref{gentrans} for ${\varphi}_k$ for $k \le 10N$, we obtain that $Q{\chi}
u_{\lambda} = \Cal O({\lambda}^{-N})$ for any $N$ and we get the
following remark.

\begin{rem}\label{transrem}
In Proposition~\ref{transprop} we may assume that ${\varphi}_k(t,x) =
{\phi}_k({\lambda}^{1/10}t, 
{\lambda}^{2/5}x) \in S(1, g_{2/5})$, $k \ge 0$, where ${\phi}_k
\in C_0^\infty$ has 
support where $|x| \ls 1$ and $|t| \ls  {\lambda}^{1/10} \le
{\lambda}^{2/5}$, ${\lambda} \ge 1$.
\end{rem}

\section{The proof of Theorem \ref{mainthm}}\label{pfsect}  % 2.9

For the proof we will need the following modification
of~\cite[Lemma 26.4.14]{ho:yellow}. Recall that $\Cal D'_{{\Gamma}} =
\set{u \in \Cal D': \wf (u) \subset {\Gamma}}$ for ${\Gamma} \subset
T^*\br^n$, and that
$\mn{u}_{(k)}$ is the $L^2$ Sobolev norm of order $k$ of $u \in
C_0^\infty$.

\begin{lem}\label{estlem}
Let 
\begin{equation}\label{estlem0}
 u_{\lambda}(x) =  {\lambda}^{(n-1){\delta}/2}\exp(i{\lambda}^{\varrho}
 {\omega}({\lambda}^{\varepsilon}x)) 
 \sum_{j=0}^M 
 {\varphi}_{j} ({\lambda}^{{\delta}}x){\lambda}^{-j{\kappa}}
\end{equation}
with ${\omega} \in C^\infty (\br^n)$ satisfying $\im {\omega} \ge 0$
and $|d {\omega}| \ge c > 0$, ${\varphi}_j \in C^\infty_0(\br^n)$,
${\lambda} \ge 1$, ${\varepsilon}$, ${\delta}$, ${\kappa}$ and
${\varrho}$ are positive such that $\varepsilon < {\delta} < {\varepsilon} +
{\varrho}$. Here ${\omega}$ and 
${\varphi}_j$ may depend on ${\lambda}$ but uniformly, and ${\varphi}_j$ has fixed
compact support in all but one of the 
variables, for which the support is bounded by $C{\lambda}^{{\delta}}$.  
Then for any integer $N$ we have 
\begin{equation}\label{estlem1}
 \mn{u_{\lambda}}_{(-N)} \le C {\lambda}^{-N({{\varepsilon} + {\varrho}})}
\end{equation}
If ${\varphi}_0(x_0) \ne 0$ and $\im {\omega}(x_0) = 0$ for some $x_0$ then
there exists $c > 0$ and ${\lambda}_0 \ge 1$ so that
\begin{equation}\label{estlem2}
  \mn{u_{\lambda}}_{(-N)} \ge c {\lambda}^{-(N+
    \frac{n}{2})({\varepsilon}+{\varrho}) + (n-1){\delta}/2} \qquad
  {\lambda} \ge {\lambda}_0
\end{equation}
Let ${\Sigma} = \bigcap_{{\lambda} \ge 1} \bigcup_j  \supp
{\varphi}_j({\lambda}\, \cdot)$ 
and let $ {\Gamma}$ be the cone generated by 
\begin{equation}\label{estlem3}
 \set{(x,\partial{\omega}(x)),\ x
   \in {\Sigma},\ \im {\omega}(x) = 0} 
\end{equation}
then for any real $m$ we find ${\lambda}^m u_{\lambda} \to 0$ in $\Cal
D'_{ {\Gamma}}$ so 
${\lambda}^m Au_{\lambda} \to 0$ in $C^\infty$ if $A$ is a 
pseudodifferential operator such that $\wf(A) \cap  {\Gamma} =
\emptyset$. The estimates are uniform if
${\omega} \in C^\infty$ uniformly with fixed lower bound on $|d\re {\omega}|$, and
${\varphi}_j \in C^\infty$ uniformly.
\end{lem}

We shall use Lemma~\ref{estlem} for $u_{\lambda}$ in~\eqref{udef}, then
${\omega}$ will be real valued and
${\Gamma}$ in~\eqref{estlem3} will be the bicharacteristic
${\Gamma}_j$ converging to a limit bicharacteristic.

\begin{proof}[Proof of Lemma \ref{estlem}]
We shall adapt the proof of~\cite[Lemma
26.4.14]{ho:yellow} to this case. By making the change of variables $y
= {\lambda}^{\varepsilon}x$ we find that 
\begin{equation}\label{utrans}
 \hat u_{\lambda}({\xi}) = {\lambda}^{(n-1){\delta}/2- n{\varepsilon}} \sum_{j=0}^{M}
 {\lambda}^{-j{\kappa}} \int
 e^{i({\lambda}^{\varrho}{\omega}(y) 
 - \w{y,{\xi}/{\lambda}^{\varepsilon}})} {\varphi}_{j}({\lambda}^{{\delta} -
 {\varepsilon}}y)\,dy  
\end{equation}
Let $U$ be a neighborhood of the projection on the second
component of the set in~\eqref{estlem3}. When
${\xi}/{\lambda}^{\varepsilon +{\varrho}} \notin
U$ then for  ${\lambda} \gg 1$ we have that
\begin{multline*}
\bigcup_j \supp {\varphi}_j({\lambda}^{\delta - \varepsilon}\cdot) \ni y\mapsto ({\lambda}^{\varrho}{\omega}(y)  -
\w{y,{\xi}/{\lambda}^{\varepsilon}})/({\lambda}^{{\varrho}} +
 |{\xi}|/{\lambda}^{\varepsilon}) \\ = ({\omega}(y)  -
\w{y,{\xi}/{\lambda}^{\varepsilon +{\varrho}}})/(1 + |{\xi}|/
{\lambda}^{{\varepsilon} +{\varrho}})
\end{multline*}
is in  
a compact set of functions with non-negative imaginary part with a fixed
lower bound on the gradient of the real part. Thus, by integrating by
part in~\eqref{utrans} we find for any positive integer $m$ that 
\begin{equation}\label{pfest}
 |\hat u_{\lambda}({\xi})| \le C_m{\lambda}^{-(n-1)\delta /2 +
   m({\delta} - {\varepsilon})}( {\lambda}^{{\varrho}} +
 |{\xi}|/{\lambda}^{\varepsilon})^{-m} \qquad
 {\xi}/{\lambda}^{\varepsilon +{\varrho}} \notin U \qquad {\lambda} \gg 1
\end{equation}
This gives any negative power of ${\lambda}$ for $m$ large enough
since ${\delta} < {\varepsilon} + {\varrho}$.
If $V$ is bounded and $0 \notin \ol V$ then since $u_{\lambda}$ is
uniformly bounded in $L^2$ we find
\begin{equation*}
 \int_{{\tau}V}  |\hat u_{\lambda}({\xi})|^2 (1  +
 |{\xi}|^2)^{-N}\,d{\xi} \le C_V{\tau}^{-2N} \qquad {\tau} \ge 1
\end{equation*}
Using this estimate with ${\tau} = {\lambda}^{{\varepsilon} + {\varrho}}$
together with the estimate~\eqref{pfest} we obtain~\eqref{estlem1}.  
If ${\chi}
\in C_0^\infty$ then we may apply~\eqref{pfest} to
${\chi}u_{\lambda}$, thus we find for any positive integer $j$ that
\begin{equation*}
  |\widehat {{\chi}u}_{\lambda}({\xi})| \le C_j{\lambda}^{-(n-1)\delta /2 +
   j({\delta} - {\varepsilon})}( {\lambda}^{{\varrho}} +
 |{\xi}|/{\lambda}^{\varepsilon})^{-j} \qquad {\xi}
  \in W \qquad {\lambda} \gg 1
\end{equation*}
if $W$ is any closed cone with $
{\Gamma} \bigcap (\supp {\chi}\times W) = \emptyset$. 
Thus we find that
${\lambda}^m u_{\lambda} \to 0$ in $\Cal D'_{ {\Gamma}}$ for every $m$.
To prove \eqref{estlem2} we assume $x_0 = 0$ and take ${\psi}\in
C_0^\infty$. If $\im {\omega}(0) = 0$ and ${\varphi}(0) \ne
0$ we find
\begin{multline*}
 {\lambda}^{n({\varepsilon} + {\varrho}) - (n-1){\delta}/2} 
 e^{-i{\lambda}^{\varrho}\re {\omega}(0)}\w{u_{\lambda},
   {\psi}({\lambda}^{{\varepsilon} + \varrho}\cdot)}\\ = \int 
 e^{i{\lambda}^{\varrho}({\omega}(x/{\lambda}^{\varrho}) - {\omega}(0))}{\psi}(x) 
 \sum_{j}{\varphi}_j(x/{\lambda}^{{\varepsilon} + \varrho - {\delta}})
 {\lambda}^{-j{\kappa}}\,dx  \\\to \int
 e^{i\w{\re \partial_x{\omega}(0),x}}{\psi}(x)
 {\varphi}_0(0)\,dx \qquad {\lambda} \to + \infty
\end{multline*}
which is not equal to zero for some suitable ${\psi} \in
C^\infty_0$. In fact, we have 
${\varphi}_j(x/{\lambda}^{{\varepsilon} + \varrho - {\delta}}) =
{\varphi}_j(0) + \Cal O({\lambda}^{{\delta} -
  {\varepsilon}-{\varrho}}) \to {\varphi}_j(0)$ when ${\lambda} \to \infty$, because
${\delta} < {\varepsilon} + {\varrho}$. Since
\begin{equation*}
 \mn{{\psi}({\lambda}^{{\varepsilon}+{\varrho}} \cdot)}_{(N)} \le C
 {\lambda}^{(N-n/2)({{\varepsilon}+{\varrho}})} 
\end{equation*}
we obtain that $0 < c \le  {\lambda}^{(N+
\frac{n}{2})({\varepsilon}+{\varrho}) - (n-1){\delta}/2}\mn{u}_{(-N)}$ which  
gives~\eqref{estlem2} and the lemma. 
\end{proof}

\begin{proof}[Proof of Theorem~\ref{mainthm}]
Assume that ${\Gamma}$ is a limit bicharacteristic of $P$.
We are going to show that~\eqref{solvest} does not hold for any
${\nu}$, $N$ and any  
pseudodifferential operator $A$ such that ${\Gamma} \cap \wf (A) =
\emptyset$. This means that there exists
approximate solutions  $0 \ne u_j \in C^\infty_0$ to $P^*u_j \cong 0$
such that  
\begin{equation}\label{solvest0}
 \mn {u_j}_{(-N)}/(\mn{P^*{u_j}}_{({\nu})} + \mn {u_j}_{(-N-n)} +
 \mn{Au_j}_{(0)}) \to \infty  \qquad \text{when $ j \to \infty$}
\end{equation}
which will contradict the local solvability of $P$ at~${\Gamma}$ by
Remark~\ref{solvrem}. 

Let ${\Gamma}_j$ be a sequence of bicharacteristics of $p$\/ that
converges to ${\Gamma} \subset \st$ and $ \lambda_j $ given by~\eqref{c4} and~\eqref{ldef}
with $ \varepsilon $ to be determined later. 
Now the conditions and conclusions are
invariant under symplectic changes of homogeneous coordinates and
multiplication by elliptic pseudodifferential operators. Thus by
Proposition~\ref{prepprop} we may assume
that the coordinates are chosen so that ${\Gamma}_j = I \times
(0,0,{\xi}_j)$ with $|{\xi}_j| = 1$, and for any $0 < {\varepsilon} < 1/3$
and $ c > 0 $
we can write $P^* = Q + R$ where ${\Gamma}_j 
\cap \wf_{\varepsilon} (R) = \emptyset$ and $Q$ has
symbol  
\begin{equation}\label{normform}
{\tau} - r(t,x,{\xi}) + q_0(t,x,{\xi}) + r_0(t,x,{\xi})
\end{equation}
when $ |\xi| \ge c\lambda_j $ and the homogeneous distance to 
${\Gamma}_j$ is less than $c |\xi|^{-\varepsilon} $. We have that $R \in  S^{1+
  {\varepsilon}}_{1-{\varepsilon},{\varepsilon}}$, $r_0 \in
S^{3{\varepsilon} - 1}_{1-{\varepsilon},{\varepsilon}}$, $q_0 \in
S^{\varepsilon}_{1-{\varepsilon},{\varepsilon}}$ is
given by~\eqref{symb}, and $r \in  
S^{1-{\varepsilon}}_{1-{\varepsilon},{\varepsilon}}$ vanishes of
second order at ${\Gamma}_j$.

Now, we may replace the norms
$\mn{u}_{(s)}$ in~\eqref{solvest0} by the norms
\begin{equation*}
 \mn u_s^2 = \mn{\w{D_x}^s u}^2 = \int \w{{\xi}}^{2s}|\hat
 u({\tau},{\xi})|^2\, d{\tau}d{\xi}
\end{equation*}
In fact, the quotient $\w{{\xi}}/\w{({\tau},{\xi})} \cong 1$ 
when $|{\tau}| \ls |{\xi}|$, thus in a conical neighborhood of
${\Gamma}$. So replacing the norms  
in the estimate~\eqref{solvest0} only changes the
constant and the operator $A$ in the estimate~\eqref{solvest}.
By using Proposition~\ref{symplprop} we may assume that
the grazing Lagrangean space $L_j(w) \equiv \set{(s,y;0,0): \ (s,y) \in
\br^{n}}$, $\forall\, w \in {\Gamma}_j$, after conjugation with a 
uniformly bounded $ C^1 $ section $ F(t) $ of Fourier integral operators,
then $\partial_x^2 r = 0$ at ${\Gamma}_j$ and  ${\Gamma}
\cap \wf_{\varepsilon}  (A) = \emptyset $.
Observe that for each ~$ t $ we find that
$ F(t) $ is uniformly continuous in local  $ H_s $ spaces, which we may use in~\eqref{solvest0} after changing~$ A $.
Also the conjugation of $ F(t) $ with the operator with symbol~\eqref{normform}
gives a uniformly bounded expansion. By changing $ A $ again, we may then replace the local $\mn{u}_{s}$
norms by the norms $\mn{u}_{(s)}$ in~\eqref{solvest0} so that we can use Lemma~\ref{estlem}.

Now, by choosing  ${\delta} = 2/5$, ${\varepsilon} = 1/10$ and ${\varrho}
= 1/10$ and using
Propositions~\ref{expprop}, \ref{omegalem}, ~\ref{transprop} and
Remark~\ref{transrem}, we can for each ${\Gamma}_j$ construct   
approximate solution $u_{{\lambda}_j}$ on the form~\eqref{udef} 
so that $Q
u_{{\lambda}_j} = \Cal O({\lambda}^k)$, for any~$k$. The real valued
phase function is equal to
is $\w{x,{\xi}_j} + {\omega}_j(t,x)$ where $ |{\xi}_j| = 1 $ and ${\omega}_j(t,x) \in
S({\lambda_j}^{-7/10}, g_{3/10})$ and the values of 
$(t,x; {\lambda}\partial_t{\omega}_j(t,x),{\lambda}({\xi}_j +
\partial_x{\omega}_j(t,x)))$ have homogeneous distance  
to the rays through  ${\Gamma}_j$ which is $ \ls \lambda^{-2/5} $ when $|x| \ls
{\lambda_j}^{-2/5}$, i.e., in $\supp u_{{\lambda}_j}$. 
Observe that if $ \lambda \gg 1 $ then we have that $ |\xi_0 + \partial_x\omega(t,x)| \ge c $ 
in  $\supp u_{{\lambda}_j}$ for some $ c > 0  $. We find that
 $$
{\omega}_j(t,x) = {\lambda_j}^{-7/10} \wt
{\omega}_j{\lambda_j}^{3/10}t,{\lambda_j}^{3/10}x)
 $$ 
where $\wt {\omega}_j
\in C^\infty$ uniformly so $\partial_x {\omega} = \Cal
O({\lambda}_j^{-2/5})$ when $x  = \Cal O({\lambda}_j^{-2/5})$ and
$$
{\lambda_j}(\w{x,{\xi}_j} + {\omega}_j(t,x)) =
{\lambda_j}^{7/10}\w{{\lambda_j}^{3/10}x,{\xi}_j} + {\lambda}^{-4/10}\wt
{\omega}_j({\lambda_j}^{3/10}t,{{\lambda_j}^{3/10}x})
$$ 
Thus ${\delta}= 2/5$, 
${\varepsilon} = 3/10$,  ${\kappa} = 1/10$ and ${\varrho} = 7/10$
in~\eqref{estlem0} so we find ${\varepsilon} + {\varrho} = 
1 > {\delta} > \varepsilon$.

The amplitude functions ${\varphi}_{k,j}(t,x) =
{\phi}_{k,j}({\lambda_j}^{2/5}t, {\lambda_j}^{2/5}x)$ where
${\phi}_{k,j} \in 
C^\infty_0$ uniformly in~$j$ with fixed compact support in~$x$, but
in~$t$ the support is bounded by
$C{\lambda_j}^{2/5}$, so $u_{{\lambda}_j}$
will satisfy the conditions in Lemma~\ref{estlem} uniformly.
Clearly differentiation of~$Qu_{\lambda_j}$ can at most give a factor
${\lambda}_j$ since ${\varepsilon} + {\varrho} = 1$ and
${\delta} < 1$. Because of the bound on the
support of $u_{\lambda_j}$ we may obtain that 
\begin{equation} \label{8.13}
\mn{Q u_{\lambda_j}}_{({\nu})} = \Cal
O({\lambda}_j^{-N-n})
\end{equation} 
for the chosen ${\nu}$. 

If $\wf (A) \bigcap {\Gamma} = \emptyset$,
then we find $\wf (A) \bigcap {\Gamma}_j = \emptyset$ for large $j$,
so Lemma~\ref{estlem} gives $\mn{Au_{{\lambda}_j}}_{(0)} 
= \Cal O({\lambda}_j^{-N-n})$ when $ j \to \infty $.
Since $x = \Cal O({\lambda}_j^{-2/5})$ in $\supp u_{{\lambda}_j}$,
we find that the values of $(t,x; {\lambda}\partial_t{\omega}_j(t,x), {\lambda}({\xi}_j +
\partial_x{\omega}_j(t,x)))$ have homogeneous distance to the rays through ${\Gamma}_j$ 
which is $ \ls  \lambda^{-2/5}$ for $x \in \supp u_{\lambda_j}$,
and this converges to ${\Gamma}$.
Thus, if $R \in S^{11/10}_{9/10,1/10}$ such that
$\wf_{1/10}(R) \bigcup {\Gamma}_j = \emptyset$ then we find from the
expansion~\eqref{trevesexp}  
that all the terms of~$Ru_{{\lambda}_j}$ vanish for large enough
~${\lambda}_j$. In fact,  
since ${\lambda}_j^{-2/5} \ll {\lambda}_j^{-1/10}$ for $j \gg 1$, we find 
for any~${\alpha} $ and~$K$ that
$$
\partial^{\alpha} R(t,x;{\lambda}_j((0,\xi_j) + \partial_{t,x} {\omega}_j(t,x))) =
O({\lambda}_j^{-K})
$$ 
in $\bigcup_k \supp {\varphi}_{k,j}$. As before, we find that
$\mn{Ru_{{\lambda}_j}}_{(\nu)} = \Cal O({\lambda}_j^{-N-n})$
by the bound on the support of~$u_{\lambda}$, so we obtain
from~\eqref{8.13} that  
\begin{equation} 
\mn{P^*u_{\lambda_j}}_{({\nu})} = \Cal
O({\lambda}_j^{-N-n})
\end{equation} 
for the chosen ${\nu}$. 

Since  ${\varepsilon} + {\varrho} =
1$ we also find from  Lemma~\ref{estlem} that
\begin{equation*}
 {\lambda}_j^{-N} = {\lambda}_j^{-N({\varepsilon}+ {\varrho})}\gs 
\mn{u_{{\lambda}_j}}_{(-N)}  \gs 
 {\lambda}_{j} ^{-(N+
    \frac{n}{2})({\varepsilon}+{\varrho}) + (n-1){\delta}/2}
 =  {\lambda}_{j} ^{-N - n/2 + (n-1)/5}   \ge {\lambda}_j^{-N-n/2}
\end{equation*}
when ${\lambda}_j \ge 1$. We obtain that~\eqref{solvest0} holds for
$u_j = u_{{\lambda}_j}$  when $j \to \infty$, 
so Remark~\ref{solvrem} gives that~$P$ is not solvable at the limit
bicharacteristic~${\Gamma}$. 
\end{proof}

\bibliographystyle{plain}

\bibliography{nec}

\end{document}